\documentclass{siamart251216}
\usepackage{amssymb}
\usepackage{booktabs,tabularx}
\usepackage{float}
\usepackage{tikz}

\definecolor{cyclefive}{RGB}{0,114,178}
\definecolor{cyclethree}{RGB}{213,94,0}
\definecolor{cyclethird}{RGB}{0,158,115}
\tikzset{
  fivecycle/.style={draw=cyclefive,line cap=round},
  threecycle/.style={draw=cyclethree,line cap=round},
  thirdcycle/.style={draw=cyclethird,line cap=round}
}

\newsiamremark{remark}{Remark}
\newsiamremark{example}{Example}

\newcommand\epi{\bar{\pi}}
\newcommand\eiota{\bar{\iota}}
\newcommand\spi{\eiota {\epi}^{-1}}
\newcommand\norm[1]{{\lvert\lvert#1\rvert\rvert}_3}

\newcommand{\Reflect}{\mathfrak{r}}

\newcommand{\Cycle}{\mathsf{C}}

\DeclareMathOperator{\Supp}{Supp}
\DeclareMathOperator{\Fix}{Fix}

\title{Twisted Bracelets for Sorting by Transpositions:
   the Transposition Diameter of $S_{16}$}
\author{Luiz A. G. Silva\thanks{Departamento de Matemática, Universidade de Brasília, Brasília, Brazil (\email{laugustogarcia@gmail.com}, \email{norai@unb.br}).}
  \and Luis A. B. Kowada\thanks{Instituto de Computação, Universidade Federal Fluminense, Niterói, Brazil (\email{luis@ic.uff.br}).}
  \and Noraí R. Rocco\footnotemark[1]
  \and Maria E. M. T. Walter\thanks{Departamento de Ciência da Computação, Universidade de Brasília, Brasília, Brazil (\email{mariaemilia@unb.br}).}}
\headers{Twisted Bracelets and Transposition Diameter}%
  {Silva, Kowada, Rocco, and Walter}

\ExplSyntaxOn
\NewDocumentCommand{\cycle}{ O{\;} m }
{
  (
    \alec_cycle:nn { #1 } { #2 }
  )
}

\seq_new:N \l_alec_cycle_seq
\cs_new_protected:Npn \alec_cycle:nn #1 #2
{
  \seq_set_split:Nnn \l_alec_cycle_seq { , } { #2 }
  \seq_use:Nn \l_alec_cycle_seq { #1 }
}
\ExplSyntaxOff

\begin{document}
\raggedbottom
{\hbadness=2500 \maketitle}

\begin{abstract}
  \textsc{Sorting By Transpositions} (SBT) seeks the minimum number of transpositions required to sort a permutation $\pi$ on $n$ symbols into the identity $\iota$. Let $N=n+1$. A cyclic-target pair $(\omega,\beta)$ consists of an even permutation $\omega$ and an $N$-cycle $\beta$ for which $\rho=\omega\beta$ is an $N$-cycle. An SBT instance is the special case $(\bar{\iota}{\bar{\pi}}^{-1},\bar{\pi})$, where $\bar{\pi}$ and $\bar{\iota}$ encode $\pi$ and $\iota$, and $\bar{\iota}{\bar{\pi}}^{-1}\bar{\pi}=\bar{\iota}$. For a prescribed fixed-point-free cycle type, fixed-content words encode $\omega$, with colors distinguishing cycles and ranks recording their orientations relative to $\beta$. A word is realizable exactly when $\rho=\omega\beta$ is an $N$-cycle. Permutations of equal-part colors and shifts of rank origins form auxiliary symmetries that, together with word rotation and position reflection coupled to rank inversion, define a twisted dihedral action. Its orbits are twisted bracelets, and its realizable orbits correspond bijectively to extended-toric equivalence classes of cyclic-target pairs, where reflection is adjoined to classical toric equivalence. The correspondence yields a Burnside identity, and our method generates one encoding word for each such class. The transposition diameter $TD(n)$ is the largest transposition distance in $S_n$. Combining fixed-point contraction and reductions of the ambient instance space based on cycle structure with exhaustive verification of every remaining twisted bracelet, we prove $TD(16)=9$, closing a twenty-five-year gap. This result also yields $TD(19)=11$ and, for every $n\equiv1\pmod{3}$ with $n\geq16$, $TD(n)\leq\left\lfloor(2n-2)/3\right\rfloor-1$, improving the previous general upper bound by one for these $n$.
\end{abstract}



\section{Introduction}

\textsc{Sorting By Transpositions} (SBT) asks for the minimum number of transpositions
required to transform a permutation $\pi\in S_n$ of $\{1,\dots,n\}$ into
the identity permutation $\iota=[1\ 2\ \cdots\ n]$.  Here a
\emph{transposition} is the operation that exchanges two adjacent blocks of
symbols, modeling a large-scale genomic rearrangement.  The
resulting minimum is the \emph{transposition distance} $d_t(\pi)$, and the
\emph{transposition diameter} is
$TD(n):=\max_{\pi\in S_n}d_t(\pi)$.

Bafna and Pevzner~\cite{BafnaPevzner1998} introduced SBT and gave the first approximation
algorithm, with ratio $1.5$, by exploiting properties of the \emph{cycle graph} structure
recalled in Appendix~\ref{appendix-a}.  Their work also established a lower bound for
$d_t(\pi)$ and the upper bound $TD(n)\leq\lfloor3n/4\rfloor$.
Eriksson et al.~\cite{eriksson2001sorting} later improved the latter to
$\lfloor(2n-2)/3\rfloor$ for $n\geq9$. Elias and Hartman~\cite{EliasHartman2006} proposed a
$1.375$-approximation based on a \emph{simplification} that embeds the input
permutation with additional symbols while preserving its lower bound.
The computational complexity of SBT remained unresolved until Bulteau,
Fertin, and Rusu~\cite{bulteau2012sorting} proved that SBT is
$\mathrm{NP}$-hard.

Silva et al.~\cite{silva2022new} showed that Elias and Hartman's
simplification can lose an initial sequence of transpositions present in
the original permutation.  As a consequence, the algorithm could perform
an extra transposition and fail to satisfy the claimed approximation guarantee.
They introduced an algebraic formulation that avoids simplification and
restores the $1.375$ ratio, with a high running time $O(n^6)$; subsequently, Alexandrino et
al.~\cite{alexandrino20221} reduced this to $O(n^5)$.  More recently, Silva, Kowada, and
Walter~\cite{silva2023barrier} proved that the classical
lower bound of Bafna and Pevzner~\cite{BafnaPevzner1998} cannot support an approximation ratio below $1.375$.

Eriksson et al.~\cite{eriksson2001sorting} determined $TD(n)$ for
$1\leq n\leq15$.  Exhaustive Cayley graphs for $n\leq10$ and filtered
searches of \emph{toric classes}---equivalence classes of permutations under
coupled cyclic shifts of positions and values---for $11\leq n\leq13$ gave
$TD(11)=6$ and $TD(12)=7$. For $n=13$, they exhibited five torically
inequivalent permutations requiring eight moves, establishing $TD(13)=8$. Their
recurrence $TD(n)\leq TD(n-3)+2$ for $n\geq9$ completed the upper bounds
for $n=14,15$; the reverse permutation $[n\ n-1\ \cdots\ 1]$ and an explicit hard-to-sort
permutation attained them, respectively. Korchmaros~\cite{korchmaros2015combinatorial}
later observed that the proof of the recurrence implicitly used invariance of
transposition distance under toric equivalence and supplied a completed
argument; the required invariance is also established in
Proposition~\ref{prop:rotation-invariance} below. Meidanis, Walter, and
Dias~\cite{MeidanisWalterDias1997} first proved that the reverse permutation
has distance $\lfloor n/2\rfloor+1$ for $n\geq3$. Elias and
Hartman~\cite{EliasHartman2006} subsequently proved the stronger lower bound
$TD(n)\geq\lfloor(n+1)/2\rfloor+1$ for every odd $n\geq13$. For even $n$,
the same inequality is the reverse-permutation lower bound, so it holds for
every $n\geq12$. Lu and Yang~\cite{LuYang2010} claimed the asymptotically
stronger lower bound $TD(n)\geq(17n+1)/33$, obtained by computation, but
Cunha et al.~\cite{Cunha2013} showed that this computation was incorrect and
set the best known lower bound back to $\lfloor(n+1)/2\rfloor+1$.
Korchmaros~\cite{korchmaros2015combinatorial}
established $TD(17)=10$ without an additional computer search by noting that
the Elias--Hartman~\cite{EliasHartman2006} lower bound and the upper bound of
Eriksson et al.~\cite{eriksson2001sorting} both equal $10$ at
$n=17$. For the intervening value,
the reverse-permutation formula of Meidanis, Walter, and
Dias~\cite{MeidanisWalterDias1997} gives $TD(16)\geq\lceil17/2\rceil=9$,
while the Eriksson recurrence~\cite{eriksson2001sorting} and $TD(13)=8$ give
$TD(16)\leq TD(13)+2=10$. Thus $9\leq TD(16)\leq10$, leaving $n=16$
the smallest unresolved case.

To close this gap, we must exclude transposition distance $10$ throughout
$S_{16}$, whose $16!\approx 2.1\times 10^{13}$ permutations make exhaustive
search at the permutation level infeasible. Even for a fixed cycle type,
generating all instances before
symmetry reduction would repeatedly construct equivalent cases. We therefore
enumerate directly modulo distance-preserving symmetries. Let $N=n+1$, and let
$A_N$ denote the alternating group on $N$ symbols. In the algebraic setting
used here, toric rotations together with reflection define
\emph{extended-toric equivalence}. We encode permutations $\omega\in A_N$ of a
prescribed fixed-point-free cycle type by \emph{fixed-content words}; each
symbol has a prescribed multiplicity. The cycles of $\omega$ are
\emph{oriented} or \emph{unoriented} according to their cyclic orders relative to a fixed
$N$-cycle $\beta$.  A completed word is \emph{realizable} precisely when
$(\omega,\beta)$ is a \emph{cyclic-target pair}, so that $\rho=\omega\beta$ is also
an $N$-cycle, thus modeling an SBT instance.  Colors identify cycles and
ranks record orientation. Permutations of equal-part colors and shifts of
rank origins form an auxiliary action. Together with word rotation and
position reflection combined with rank inversion, these symmetries define
a twisted dihedral action. Its orbits are \emph{twisted bracelets}.

Hultman~\cite{hultman1999toric} counted ordinary toric classes and
introduced what he called \emph{Pevzner cycles}.  He showed this structure to
be equivalent to cycle graph~\cite{BafnaPevzner1998} and invariant under toric equivalence, but did
not count toric classes by their number of cycles.  Doignon and
Labarre~\cite{doignon2007hultman} subsequently counted permutations by their
number of cycles in the associated cycle graph and, more generally, by the
prescribed lengths of those cycles, without quotienting by toric equivalence.
Thus neither result counts toric classes of prescribed cycle type.
For every fixed-point-free cycle type, our orientation annotations partition
the set counted by Doignon and Labarre before the dihedral quotient is taken.

Our main contribution is this twisted-bracelet framework. Its orbit
correspondence identifies, for each fixed-point-free cycle type, with the orientations of its cycles specified relative to \(\beta\), the realizable twisted bracelets with the
extended-toric classes of cyclic-target pairs. Burnside's lemma therefore yields
a counting identity for these classes. We also give a direct-generation method
that emits one fixed-content encoding word for each class. The method is refined by cycle
lengths and every allowable cycle-orientation annotation; to the best of
our knowledge, such a refinement is new in SBT.

We apply the framework to the $33$
fixed-point-free cycle types in $A_{17}$. Fixed-point contraction and
reductions based on cycle structure settle most of the resulting cases; for the remainder,
we submit the generated representatives to exhaustive sorting searches. The
computer-assisted part is divided into witness emission and certificate
checking. A GPU-accelerated witness emitter stores every sorting sequence on
which the proof relies in a persistent RocksDB database~\cite{rocksdb} that serves as a
certificate, indexed by the corresponding generated word. A separate CPU
certificate checker independently reconstructs the division into cases and
verifies that it is exhaustive. It also reconstructs the required
representatives, reads the certificate without invoking the search procedure,
and replays every stored sequence exactly. The two programs have
separate implementations of the algebra and enumeration; the certificate
format is their only interface.
Together, these steps establish $TD(16)=9$, closing a twenty-five-year gap. This result
also settles the next odd case, $TD(19)=11$, and improves the general upper
bound of Eriksson et al.~\cite{eriksson2001sorting} by one for every
$n\equiv1\pmod{3}$ with $n\geq16$.
Source code and
reproduction instructions, including the complete source code for both
computations, are publicly available in a GitHub
repository~\cite{sourcecode}, and the emitted certificate is available for
direct download~\cite{td16certificate}.

The paper is organized as follows.  Section~\ref{sec:pg} fixes the
permutation-group conventions, and Section~\ref{sec:comb-prelim} collects
the combinatorial background on fixed-content words, group actions, and
orbit counting. Section~\ref{sec:preliminaries} gives the algebraic
formulation of SBT.  Section~\ref{toric} formulates toric
equivalence with reflection and extends it to cyclic targets.
Section~\ref{sec:twisted-bracelets} develops the twisted-bracelet
framework using fixed-content words. Section~\ref{sec:td16} presents
the exhaustive proof of $TD(16)=9$, derives $TD(19)=11$, and improves the
upper bound for $n\equiv1\pmod{3}$ as corollaries,
and Section~\ref{sec:conclusion}
concludes the paper.

\section{Permutation groups}\label{sec:pg}

The following conventions and standard facts may be found in abstract
algebra textbooks~\cite{dummit2004abstract,Gallian2009,Herstein}.  The
algebraic representation used later depends only on elementary properties
of permutations, conjugation, and centralizers.

A \emph{group} is a set $G$ equipped with an associative binary operation,
an identity element, and an inverse for every element of $G$.  It is
\emph{finite} when $G$ has finitely many elements.  Let $E$ be a finite
set of $n$ symbols.  A \emph{permutation} of $E$ is a bijection from $E$
to itself, and the set of all such permutations forms the \emph{symmetric
  group} $S_n$ under composition.  Thus, for $\delta,\epsilon\in S_n$, the
product $\delta\epsilon$ is the permutation defined by
$(\delta\epsilon)(x)=\delta(\epsilon(x))$; in particular, the rightmost factor
acts first.  A \emph{subgroup} of $S_n$ is a subset that is itself a group
under composition.  Throughout the paper, we identify $E$ with the
standard set $\{1,2,\dots,n\}$.

An element $x \in E$ is a \emph{fixed point} of $\delta$ if $\delta(x) = x$.
If there exist distinct elements $c_1, \dots, c_\kappa \in E$ such that
$\delta(c_i) = c_{i+1}$ for $1 \le i < \kappa$, $\delta(c_\kappa) = c_1$,
and all other elements are fixed, then $\delta$ is called a
\emph{$\kappa$-cycle}. Its \emph{length} is $\kappa$. In \emph{cycle notation},
such a permutation is written $(c_1\ c_2\dots c_\kappa)$, and any cyclic
rotation represents the same cycle.

Although a $2$-cycle is often called a transposition in the algebraic
literature, this paper uses the term transposition exclusively for the
operation that swaps two adjacent blocks in a permutation representing gene
order.

The \emph{support} of $\delta$, denoted $\Supp(\delta)$, is the set of elements
moved by $\delta$, i.e., $\Supp(\delta) = \{ x \in E\ |\ \delta(x) \ne x \}$. Two
permutations $\delta, \epsilon \in S_n$ are said to be \emph{disjoint} if $\Supp
  (\delta) \cap \Supp(\epsilon) = \emptyset$. In that case, $\delta$ and $\epsilon$
commute, so that $\delta \epsilon = \epsilon \delta$, since each element of $E$ is moved
by at most one of the two.

Every permutation $\delta \in S_n$ can be written as a product of pairwise
disjoint cycles (referred to throughout as the \emph{cycles of} $\delta$), and
this factorization is unique up to rearrangement of its factors. This is called
the \emph{disjoint cycle decomposition} of $\delta$. By convention, $1$-cycles
(i.e., fixed points) are often omitted from
the notation, though they may be included when relevant. The identity permutation,
denoted by $\iota$, fixes every element of $E$. It therefore has empty support and
a disjoint cycle decomposition consisting entirely of $1$-cycles. The \emph{cycle
  type} of a permutation is the multiset of lengths of the cycles in its disjoint
cycle decomposition, including any $1$-cycles. Throughout the paper, we identify
this multiset with the corresponding integer partition of $n$, written in weakly
decreasing order as a parenthesized tuple. Within such a tuple, $k^a$
abbreviates $a$ occurrences of the part $k$.

\begin{example}
  The permutation $(1\ 4\ 2)(3\ 6)(5\ 7) \in S_8$ has cycle type $(3,2^2,1)$.
\end{example}

For $\delta\in S_n$, the \emph{parity} of $\delta$ is the parity of the number of
$2$-cycles in any factorization of $\delta$ into $2$-cycles. This
quantity is well-defined, i.e., all such decompositions involve either
only even or only odd numbers of $2$-cycles. A permutation is \emph{even} (\emph{odd})
if it can be written as a product of an even (odd) number of $2$-cycles.
It is a standard fact that an $n$-cycle is an even (odd)
permutation when $n$ is odd (even).
Moreover, the product of two permutations
with the same parity is always even. The set of even permutations in
$S_n$ forms a subgroup, called the \emph{alternating group}, and is
denoted by $A_n$.  It follows from the disjoint cycle decomposition that a
permutation belongs to $A_n$ if and only if its cycle type contains an even
number of even parts.

Given $\delta, \epsilon \in S_n$, the \emph{conjugate} of $\delta$ by $\epsilon$ is the
permutation $\epsilon \delta \epsilon^{-1}$ denoted $\delta^\epsilon$. Conjugation affects
the cycles of a permutation $\delta$ by renaming each symbol $x$ as $\epsilon(x)$,
i.e., if $\delta$ maps $x \mapsto y$, then the conjugate $\delta^\epsilon = \epsilon \delta \epsilon^{-1}$
maps $\epsilon(x) \mapsto \epsilon(y)$.

\begin{example}
  Let $\delta = (1\ 2\ 3)(4\ 5), \epsilon = (1\ 4\ 6)(2\ 3) \in S_6$. Then, the conjugate of $\delta$ by $\epsilon$ is the permutation $\delta^\epsilon = (\epsilon(1)\ \epsilon(2)\ \epsilon(3))(\epsilon(4)\ \epsilon(5)) = (4\ 3\ 2)(6\ 5)$.
\end{example}

In particular, conjugation preserves cycle type, and two permutations in
$S_n$ are conjugate if and only if they have the same cycle type.

The \emph{centralizer} of a permutation $\delta \in S_n$ is the set $C_{S_n}(\delta) = \{ \gamma \in S_n\ |\ \gamma \delta = \delta \gamma \}$ of all
permutations that commute with $\delta$. Notably, if $\delta$ is an $n$-cycle,
then $C_{S_n}(\delta)$ is the \emph{cyclic group} generated by $\delta$, i.e.,
$\langle \delta \rangle=\{\iota,\delta,\delta^2,\dots,\delta^{n-1}\}$, where
powers are taken under composition and $\delta^0 = \iota$ is the identity.

\section{Combinatorial background}\label{sec:comb-prelim}

We review the group actions, fixed-content words, and orbit-counting tools
used in Section~\ref{sec:twisted-bracelets}. Standard
references include Stanley~\cite{Stanley2012}, Keller and
Trotter~\cite{KellerTrotter2017}, and Kerber~\cite{Kerber1999}.

For a finite set $Z$, write $|Z|$ for its number of elements. Let $G$ be a
finite group with identity element $e$, and let $X$ be a
finite set. An \emph{action} of $G$ on $X$ is a map $G \times X \to X$,
$(g, x) \mapsto g \cdot x$, such that $e \cdot x = x$ and
$g \cdot (h \cdot x) = (gh) \cdot x$ for every $g, h \in G$ and every
$x \in X$. The \emph{orbit} of $x \in X$ is the set
$G \cdot x = \{ g \cdot x \mid g \in G \}$. The \emph{quotient} of $X$
by $G$ is the set of its orbits, $X/G = \{ G \cdot x \mid x \in X \}$. The
\emph{fixed set} of $g \in G$ is
$\Fix_X(g) = \{ x \in X \mid g \cdot x = x \}$, i.e., the elements of
$X$ left invariant by $g$.

An action is \emph{free} when $\Fix_X(g)=\emptyset$ for every $g\ne e$.
If $G$ acts on both $X$ and $Y$, a map $f:X\to Y$ is
\emph{$G$-equivariant} when $f(g\cdot x)=g\cdot f(x)$ for every $g\in G$ and
$x\in X$. If $H$ and $K$ are groups of permutations of the same set $X$, we
say that $K$ \emph{normalizes} $H$ when $khk^{-1}\in H$ for every $k\in K$
and $h\in H$.

Fix a finite set $\mathcal{C}$, called an \emph{alphabet}, and an integer
$N \ge 3$. A
\emph{word} of length $N$ over $\mathcal{C}$ is a sequence
$\mathbf w = w_0 w_1 \dots w_{N-1}$ with $w_j \in \mathcal{C}$ for each $j$.
Given a prescribed nonnegative multiplicity for each letter, with the
multiplicities summing to $N$, a word has that \emph{content} when every
letter occurs with its prescribed multiplicity. Such words are
\emph{fixed-content words}. Their positions are
indexed by the additive cyclic group
$\mathbb{Z}_N=\{0,1,\dots,N-1\}$, whose operation is addition modulo $N$;
all subscripts below are understood modulo $N$. For
$\nu\in\mathbb Z_N$, let $r_\nu$ act on words by the position rotation
$(r_\nu\cdot\mathbf w)_j=w_{j+\nu}$.
For any set $Y$, let $\operatorname{id}_Y$ denote the identity map on $Y$;
we omit the subscript when the domain is clear.
The rotations $r_\nu$, for $\nu\in\mathbb Z_N$, form a cyclic group of order
$N$.  We use the presentation
$D_N=\langle r_1,s:r_1^N=s^2=\operatorname{id},sr_1s=r_1^{-1}\rangle$
for the \emph{dihedral group} of order $2N$, where $r_\nu=r_1^\nu$.  The
action of the involution $s$ is defined in
Section~\ref{sec:twisted-bracelets}; it reverses positions and applies an
involution to the ranks attached to symbols of oriented colors.

Fix a total order on the alphabet and extend it lexicographically to words. A
\emph{necklace} is an orbit of fixed-content words under position
rotation~\cite{Sawada2003}. A \emph{bracelet} is an orbit under position
rotation and ordinary reversal~\cite{Karim2013}. Each orbit has a unique least
word, called its \emph{representative}.

The following orbit-counting identity is also used in
Section~\ref{sec:twisted-bracelets}. It is classical; see,
for example, Keller and Trotter~\cite{KellerTrotter2017}.

\begin{lemma}[Burnside lemma]\label{lem:burnside}
  Let $G$ be a finite group acting on a finite set $X$. Then the number of
  orbits is $|X / G| = \frac{1}{|G|} \sum_{g \in G} |\Fix_X(g)|$.
\end{lemma}

\section{Preliminaries}\label{sec:preliminaries}

In genome rearrangement studies, the order of $n$ distinct genes along a linear
chromosome is commonly modeled by a permutation. Given genes labeled $1$ through
$n$, a chromosome can be represented by a sequence $\pi = [\pi_1\; \pi_2 \dots \pi_n]$, where each $\pi_i$ denotes the gene at position $i$.

Following Mira et al.~\cite{Mira2008}, let $E_N:=\{0,1,\dots,n\}$ and
$N:=n+1$, and set $\pi_0:=0$.  Thus $|E_N|=N$, and the linear permutation
$\pi$ is represented by the anchored $N$-cycle $\epi=(\pi_0\ \pi_1\ \dots\ \pi_n)
  =(0\ \pi_1\ \dots\ \pi_n)\in S_N$.
The anchor gives each linear permutation a unique cyclic writing beginning
with $0$.  The sorted permutation is represented by
$\eiota=(0\ 1\ \dots\ n)$, which is distinct from the identity
permutation $\iota=(0)(1)\cdots(n)$.  Henceforth SBT statements use the
augmented domain $E_N$, whereas the
general group-theoretic statements above use an arbitrary $n$-element set.

A $3$-cycle $\tau=(\pi_i\;\pi_j\;\pi_k)$ is said to be
\emph{applicable} to $\epi$ if the symbols $\pi_i$, $\pi_j$, and $\pi_k$
appear in $\epi$ in the same cyclic order as in $\tau$, that is,
$\epi=(\pi_i\;\dots\;\pi_j\;\dots\;\pi_k\;\dots)$.
The application of $\tau$ to $\epi$ means left multiplication of $\epi$ by
$\tau$.  Thus, and only in this case, the product $\tau\epi$ is an $N$-cycle
such that the symbols between $\pi_i$ and $\pi_{j-1}$, including $\pi_i$ but
not $\pi_j$, are cut and then pasted between $\pi_{k-1}$ and $\pi_k$, thereby
simulating a transposition:
\begin{equation*}
  \begin{split}
    \tau\epi
      &=(\pi_i\;\pi_j\;\pi_k)
        (\pi_0\;\pi_1\;\dots\;\pi_{i-1}\;\pi_i\;\dots\;\pi_{j-1}\;
         \pi_j\;\dots\;\pi_{k-1}\;\pi_k\;\dots\;\pi_n)\\
      &=(\pi_0\;\pi_1\;\dots\;\pi_{i-1}\;\pi_j\;\dots\;\pi_{k-1}\;
         \pi_i\;\dots\;\pi_{j-1}\;\pi_k\;\dots\;\pi_n).
  \end{split}
\end{equation*}

\begin{lemma}[Applicability criterion]\label{lem:applicability}
  Let $\beta$ be an $N$-cycle and $\tau$ a $3$-cycle.  The product
  $\tau\beta$ is an $N$-cycle if and only if $\tau$ is applicable to
  $\beta$; otherwise $\tau\beta$ has three cycles.
\end{lemma}

\begin{proof}
  Write $\beta=(a\ A\ b\ B\ c\ C)$, where $A$, $B$, and $C$ are possibly empty
  strings.  Direct multiplication
  gives $(a\ b\ c)\beta=(a\ A\ c\ C\ b\ B)$ and $(a\ c\ b)\beta=(a\ A)(b\ B)(c\ C)$.
  These are the two possible cyclic orders of a $3$-cycle on
  $\{a,b,c\}$, proving the claim.
\end{proof}

\begin{example}
  Let $\epi = (0\ 8\ 6\ 5\ 7\ 4\ 3\ 2\ 1)$. The $3$-cycle $\tau = (5\ 7\ 6)$ is applicable since $5$, $7$, and $6$ occur in this cyclic order in $\epi$, yielding $\tau \epi = (0\ 8\ 5\ 6\ 7\ 4\ 3\ 2\ 1)$, an $N$-cycle. In contrast, $\tau^{-1} = (5\ 6\ 7)$ is not applicable, as $\tau^{-1} \epi = (0\ 8\ 7\ 4\ 3\ 2\ 1)(5)(6)$.
\end{example}

Given an $N$-cycle $\epi$, the \textsc{Sorting by Transpositions} (SBT)
problem consists of finding the minimum number $t$, denoted $d_t(\epi)$,
of transpositions represented as applicable $3$-cycles required to transform
$\epi$ into the target $N$-cycle $\eiota$.  This algebraic distance equals
the transposition distance of the original linear
permutation, and we therefore set $d_t(\pi):=d_t(\epi)$.  Thus
\begin{align}\label{eq:sorting-identity}
  \tau_t \dots \tau_1 \epi & = \eiota,
\end{align}
where each intermediate product $\tau_i \dots \tau_1 \epi$, for all $1 \leq i \leq t$, remains an $N$-cycle. Multiplying both sides of Equation~\eqref{eq:sorting-identity} on the right by $\epi^{-1}$ gives
\begin{align}
  \tau_t \dots \tau_1 & = \eiota \epi^{-1}.\label{eq:sorting-factorization}
\end{align}
Hence the product of the $3$-cycles that
sorts $\epi$ equals $\spi$, which is an even permutation since $\epi$ and
$\eiota$ have the same parity.

For $x\in\{0,1,\dots,n\}$, with addition understood modulo
$N$, an occurrence of $x$ immediately followed by $x+1$ in the
cyclic order of $\epi$ is called an
\emph{adjacency}.  Contracting such an
adjacency means regarding its two consecutive symbols as a single
symbol and renaming the symbols of the resulting permutation.

\begin{proposition}\label{prop:contraction-preserves-distance}
  Let $\eta$ be obtained from $\epi$ by contracting one
  adjacency.  Then $d_t(\eta)=d_t(\epi)$.
\end{proposition}

\begin{proof}
  Any sorting sequence for $\eta$ lifts to one for $\epi$ by expanding the
  contracted symbol to the original adjacency and treating it as an
  indivisible block.  Hence $d_t(\epi)\leq d_t(\eta)$.  Conversely, let the
  contracted adjacency be $x,x+1$.  Delete $x+1$ from every intermediate
  cycle in a sorting sequence for $\epi$ and apply the contraction's renaming.
  Each move induces either a transposition of the resulting shorter cycle
  or the identity, if an exchanged block becomes empty. Removing the identity
  steps yields a sorting sequence for $\eta$, so $d_t(\eta)\leq d_t(\epi)$.
\end{proof}

Repeated application of Proposition~\ref{prop:contraction-preserves-distance}
shows that contracting $r$ adjacencies reduces an instance in $S_N$ to
one in $S_{N-r}$ without changing its transposition distance.

\begin{proposition}[Adjacencies and fixed points]
  \label{prop:adjacency-fixed-point}
  For $x\in E_N$, the adjacency $x,x+1$ occurs in $\epi$ if and only if
  $x+1$ is a fixed point of $\spi$.
\end{proposition}

\begin{proof}
  With addition modulo $N$, we have $\spi(x+1)=x+1$ if and only if
  $\epi^{-1}(x+1)=x$, which holds if and only if $\epi(x)=x+1$.
\end{proof}

Mira and Meidanis~\cite{mira2005algebraic} introduced the \emph{$3$-norm} on even
permutations. For $\alpha \in A_n$, $\norm{\alpha}$ is the minimum $\ell \ge 0$
such that $\gamma_{\ell}\dots\gamma_2\gamma_1=\alpha$, where each $\gamma_i$, for $1 \le i \le \ell$, is a $3$-cycle. Let $c^\circ_{\mathrm{odd}}(\alpha)$ denote the
number of odd-length cycles of $\alpha$, including $1$-cycles. They proved the
following lemma.

\begin{lemma}[Mira and Meidanis~\cite{mira2005algebraic}]\label{lem:norm}
  For every $\alpha \in A_n$, $\norm{\alpha} = \frac{n-c^\circ_{\mathrm{odd}}(\alpha)}{2}.$
\end{lemma}

Every sorting sequence gives, by
Equation~\eqref{eq:sorting-factorization}, a factorization of $\spi$ into
$d_t(\epi)$ $3$-cycles.  Lemma~\ref{lem:norm} therefore yields the
following lower bound.

\begin{lemma}[Mira and Meidanis~\cite{mira2005algebraic}]\label{lem:lb-norm}
  For every $N$-cycle $\epi \in S_N$,
  $d_t(\epi) \geq \norm{\eiota\epi^{-1}}
  = \frac{N-c^\circ_{\mathrm{odd}}(\eiota\epi^{-1})}{2}$.
\end{lemma}

\subsection{Cycle orientation and transposition moves}\label{interaction}

Let $\Cycle$ be a cycle of $\spi$. A triple $(a,b,c)$ of distinct elements of
$\Supp(\Cycle)$ is an \emph{oriented triplet} when its cyclic order in $\Cycle$ is
opposite to its cyclic order in $\epi^{-1}$; that is, when
$\Cycle=(a\dots b\dots c\dots)$ and $\epi^{-1}=(a\dots c\dots b\dots)$. We call $\Cycle$
\emph{oriented} if it contains at least one oriented triplet, and
\emph{unoriented} otherwise.

For a fixed $\epi$, the \emph{annotated part} of a cycle of $\spi$ of length
$k$ is $k_{\mathrm{o}}$ when the cycle is oriented and $k_{\mathrm{u}}$ when it
is unoriented. The \emph{annotated cycle type} of $\spi$ is the parenthesized
tuple of its annotated parts, ordered first by decreasing cycle length and,
for equal lengths, with $\mathrm{o}$ before $\mathrm{u}$. Within such a tuple,
$k_{\varepsilon}^a$ abbreviates $a$ occurrences of the annotated part
$k_{\varepsilon}$. This ordering gives a canonical representation and does
not distinguish cycles having the same annotated part.

\begin{example}
  Let
  $\epi = \cycle[\allowbreak\;]{0, 6, 5, 3, 2, 1, 8, 7, 4, 9, 14, 13, 12, 11, 10}$.
  Then $\spi=(0\ 11\ 13)\allowbreak(1\ 3\ 6)\allowbreak
  (2\ 4\ 8)\allowbreak(5\ 7\ 9)\allowbreak(10\ 12\ 14)$.
  In $\epi^{-1}$, the support symbols of $(5\ 7\ 9)$ appear cyclically as
  $5,9,7$, so $(5,7,9)$ is an oriented triplet.  The support symbols of the
  other four $3$-cycles appear cyclically as $0,11,13$; $1,3,6$; $2,4,8$;
  and $10,12,14$, respectively.
  Those cycles are therefore unoriented. Hence the annotated cycle type is
  $(3_{\mathrm{o}},3_{\mathrm{u}}^4)$.
\end{example}

The factorization in Equation~\eqref{eq:sorting-factorization} has the
equivalent form
\begin{equation*}
  \iota = \spi\tau_1^{-1}\dots\tau_t^{-1}.
\end{equation*}
Consequently, sorting $\epi$ can be viewed as successively
right-multiplying its current algebraic permutation by inverses of
applicable $3$-cycles.  For an applicable $3$-cycle $\tau$, define
\begin{equation*}
  \Delta c^\circ_{\mathrm{odd}}(\spi,\tau)
  :=c^\circ_{\mathrm{odd}}(\spi\tau^{-1})
    -c^\circ_{\mathrm{odd}}(\spi).
\end{equation*}

Meidanis, Dias, and Mira~\cite{MeidanisDias2000,mira2005algebraic}
proved the following classification.

\begin{proposition}[Move types]\label{prop:move-types}
  If $\tau$ is an applicable $3$-cycle, then
  \[
    \Delta c^\circ_{\mathrm{odd}}(\spi,\tau)\in\{-2,0,2\}.
  \]
\end{proposition}

An applicable $3$-cycle $\tau$ is called a $\mu$-move if $\Delta c^\circ_{\mathrm{odd}}(\spi,\tau)=\mu$. A \emph{qualifying $2$-move} is a $2$-move that creates a fixed point
in the resulting algebraic permutation.

\begin{corollary}[Oriented $3$-cycles]
  \label{cor:oriented-three-cycles}
  If $\Cycle$ is an oriented $3$-cycle of $\spi$, then $\tau=\Cycle$ is
  a qualifying $2$-move that fixes the three symbols of $\Cycle$.
\end{corollary}

\begin{proof}
  The cyclic order of the support of $\Cycle$ in $\Cycle$ is opposite to that
  in $\epi^{-1}$ and hence agrees with that in $\epi$.
  Thus $\tau=\Cycle$ is applicable.  Right multiplication of $\spi$ by
  $\Cycle^{-1}$ replaces $\Cycle$ by three fixed points.  The number of
  odd-length cycles therefore increases from one to three, so $\tau$ is a
  $2$-move and is qualifying.
\end{proof}

\begin{proposition}[Even-pair reduction]
  \label{prop:even-pair-reduction}
  If $\spi$ contains an even-length cycle, then it contains two such cycles
  and admits a qualifying $2$-move. More precisely, let $\Cycle$ and
  $\Cycle'$ be any two even-length cycles of $\spi$. There are consecutive
  symbols $a,b$ in $\Cycle$ and a symbol $c\in\Supp(\Cycle')$ such that
  $\tau=(a\ b\ c)$ is applicable. Right multiplication by $\tau^{-1}$
  replaces $\Cycle$ and $\Cycle'$ by the fixed point $(b)$ and a cycle of
  length $\lvert\Cycle\rvert+\lvert\Cycle'\rvert-1$, leaving every other
  cycle unchanged.
\end{proposition}

\begin{proof}
  If $e$ is the number of even-length cycles of $\spi$, then
  $\operatorname{sgn}(\spi)=(-1)^e$. Since $\spi$ is even, $e$ is even, so
  the existence of one even-length cycle guarantees a second. Choose
  $c\in\Supp(\Cycle')$ and linearly order the other symbols by reading $\epi$
  immediately after $c$ until returning to $c$. Let $a$ be the first symbol
  of $\Supp(\Cycle)$ in this order, and let $b$ be its successor in $\Cycle$.
  Then $a$ precedes $b$ in the linear order, so $a,b,c$ occur in this cyclic
  order in $\epi$. Hence $\tau=(a\ b\ c)$ is applicable.

  Write $\Cycle=(a\ b\ A)$ and $\Cycle'=(c\ B)$, where $A$ and $B$ are
  strings containing the remaining symbols of the respective cycles. A
  direct multiplication gives
  \[
    \Cycle\Cycle'\tau^{-1}=(a\ B\ c\ A)(b).
  \]
  Both original cycles have even length, whereas the new nontrivial cycle
  has odd length. The number of odd-length cycles therefore increases by
  two, so $\tau$ is a $2$-move; the fixed point $(b)$ makes it qualifying.
\end{proof}

\section{Extended-toric equivalence and cyclic targets}\label{toric}

Toric equivalence~\cite{eriksson2001sorting} identifies permutations related
by coupled cyclic shifts of positions and values. Hausen et
al.~\cite{Hausen2010} studied its unitary classes and their relationship with
SBT, and Korchmaros~\cite{korchmaros2015combinatorial} established invariance of
transposition distance within toric classes. In the anchored-cycle
representation, a toric shift adds a constant modulo $N$ to every symbol and
then rewrites the cycle beginning at $0$.  Korchmaros
also combined toric maps with a reverse map on one-line permutations and showed
that they generate a dihedral group of Cayley-graph automorphisms.  Because
that one-line formalism is rather different from the anchored-cycle
representation used here, we define the rotation and reflection directly
on $N$-cycles and prove their group relations and distance invariance.  The
resulting action identifies further $N$-cycles equivalent to $\epi$ and supplies the
symmetry from which the fixed-content encoding below is constructed.

\subsection{Toric rotations}

For $\nu\in\mathbb Z_N$, define
$R_\nu(\epi):=\epi^{\eiota^\nu}$.
Since $C_{S_N}(\eiota)=\langle\eiota\rangle$, the \emph{toric equivalence
  class} of $\epi$ is
\[
  \epi^\circ
  :=\{R_\nu(\epi):\nu\in\mathbb Z_N\}
  =\{\epi^\upsilon:\upsilon\in C_{S_N}(\eiota)\}.
\]

\begin{example}
  Let $\epi=\cycle{0,3,1,2}$ and
  $\eiota=\cycle{0,1,2,3}$.  Then
  \[
    C_{S_4}(\eiota)
    =\{\iota,(0\ 1\ 2\ 3),(0\ 2)(1\ 3),(0\ 3\ 2\ 1)\},
  \]
  and the toric equivalence class is
  \[
    \epi^\circ
    =\{\cycle{0,3,1,2},\cycle{0,2,3,1},
    \cycle{0,2,1,3},\cycle{0,1,3,2}\}.
  \]
\end{example}

\begin{proposition}[Rotation invariance]
  \label{prop:rotation-invariance}
  For every $\nu\in\mathbb Z_N$, if $\tau_1,\ldots,\tau_s$ is a sorting
  sequence for $\epi$, then
  $\tau_1^{\eiota^\nu},\ldots,\tau_s^{\eiota^\nu}$ is a sorting sequence for
  $R_\nu(\epi)$.  In particular, $d_t(R_\nu(\epi))=d_t(\epi)$.
\end{proposition}

\begin{proof}
  Conjugation by $\eiota^\nu$ sends every intermediate cycle according to
  \begin{equation*}
  (\tau_i\cdots\tau_1\epi)^{\eiota^\nu}
  =\tau_i^{\eiota^\nu}\cdots\tau_1^{\eiota^\nu}R_\nu(\epi),
  \end{equation*}
  and the resulting permutation remains an $N$-cycle.  Conjugation also fixes
  the target $\eiota$.
  Thus the conjugated moves form a sorting sequence of the same length.
  Conjugation by $(\eiota^\nu)^{-1}$ gives the inverse correspondence, so the
  distances are equal.
\end{proof}

\subsection{Reflection}

Define the involution $\Reflect\in S_N$ by
$\Reflect(x):=-x\pmod N$ for $x\in E_N$.
Thus $\Reflect$ fixes $0$ and reverses the cyclic order of the remaining
symbols.
For every $N$-cycle $\epi$, define $F(\epi):=(\epi^{-1})^{\Reflect}$.
The identities $\Reflect^2=\iota$ and
$\eiota^{\Reflect}=\eiota^{-1}$ imply $F(\eiota)=\eiota$.

\begin{proposition}[Extended toric action]
  \label{prop:extended-toric-action}
  For $\nu,\xi\in\mathbb Z_N$, we have
  $R_\nu\circ R_\xi=R_{\nu+\xi}$,
  $F^2=\operatorname{id}$, and $F\circ R_\nu=R_{-\nu}\circ F$.
  Hence the maps $R_\nu$ and $F$ define an action of the dihedral group $D_N$
  on the set of $N$-cycles.
\end{proposition}

\begin{proof}
  The first relation follows by composing conjugations by powers of
  $\eiota$.  For every $N$-cycle $\epi$, the second relation follows from
  \[
    F^2(\epi)
    =\left(\left((\epi^{-1})^{\Reflect}\right)^{-1}\right)^{\Reflect}
    =\left((\epi^{-1})^{-1}\right)^{\Reflect^2}
    =\epi^{\Reflect^2}
    =\epi.
  \]
  Finally,
  $\Reflect\eiota^\nu=(\eiota)^{-\nu}\Reflect$, which gives the third relation.
  These are the defining rotation--reflection relations of $D_N$.
\end{proof}

The \emph{extended-toric equivalence class} of $\epi$ is its $D_N$-orbit,
namely
\[
  {\epi^\circ}_{\mathrm{ext}}
  :=\{R_\nu(\epi),F(R_\nu(\epi)):\nu\in\mathbb Z_N\}
  =\left\{\epi^\upsilon,
  \left((\epi^\upsilon)^{-1}\right)^{\Reflect}
  :\upsilon\in C_{S_N}(\eiota)\right\}.
\]

\begin{example}
  Let $\epi=\cycle{0,3,5,4,1,2}$. The toric equivalence class of $\epi$ is $\epi^\circ=\{\cycle{0,3,5,4,1,2},\allowbreak
    \cycle{0,5,2,3,1,4},\allowbreak
    \cycle{0,3,4,2,5,1},\allowbreak
    \cycle{0,2,1,4,5,3},\allowbreak
    \cycle{0,4,1,3,2,5},\allowbreak
    \cycle{0,1,5,2,4,3}\}$, whereas ${\epi^\circ}_{\mathrm{ext}}=\epi^\circ \cup \{\cycle{0,4,5,2,1,3},\allowbreak
    \cycle{0,2,5,3,4,1},\allowbreak
    \cycle{0,5,1,4,2,3},\allowbreak
    \cycle{0,3,1,2,5,4},\allowbreak
    \cycle{0,1,4,3,5,2},\allowbreak
    \cycle{0,3,2,4,1,5}\}$
\end{example}

\begin{proposition}[Reflection preserves sorting]
  \label{prop:reflection-preserves-sorting}
  Let $\tau$ be applicable to $\epi$, and set
  $\epi'=\tau\epi$.  Define
  $\widehat\tau:=\left((\tau^{-1})^{\epi^{-1}}\right)^{\Reflect}$.
  Then $\widehat\tau$ is applicable to $F(\epi)$ and
  $F(\epi')=\widehat\tau F(\epi)$.
  Consequently, reflection bijects sorting sequences of $\epi$ and
  $F(\epi)$ without changing their lengths, and
  $d_t(F(\epi))=d_t(\epi)$.
\end{proposition}

\begin{proof}
  Since
  $\epi^{-1}\tau^{-1}=(\tau^{-1})^{\epi^{-1}}\epi^{-1}$,
  \[
    F(\epi')
    =(\epi^{-1}\tau^{-1})^{\Reflect}
    =\left((\tau^{-1})^{\epi^{-1}}\right)^{\Reflect}
    (\epi^{-1})^{\Reflect}
    =\widehat\tau F(\epi).
  \]
  The move $\widehat\tau$ is a $3$-cycle, and both $F(\epi)$ and
  $F(\epi')$ are $N$-cycles.  It is therefore applicable to $F(\epi)$ by
  Lemma~\ref{lem:applicability}.  Applying this identity at each step of a
  sorting sequence gives a sequence ending at
  $F(\eiota)=\eiota$.  Since $F$ is an involution, this correspondence is a
  bijection and the distances are equal.
\end{proof}

\subsection{Cyclic targets}\label{sec:cyclic-targets}

A \emph{cyclic-target pair} on $E_N$ is a pair $(\omega,\beta)$ in which
$\beta$ and $\rho:=\omega\beta$ are $N$-cycles. We call $\beta$ the
\emph{current cycle}, $\rho$ the \emph{target cycle}, and the necessarily even permutation
$\omega=\rho\beta^{-1}$ the \emph{algebraic permutation}. A cyclic-target pair
with target $\eiota$ is an \emph{ordinary SBT pair}.

\begin{proposition}[Renaming to the ordinary target]
  \label{prop:ordinary-target-renaming}
  Every cyclic-target pair $(\omega,\beta)$ admits a renaming
  $\upsilon\in S_N$ such that $\rho^\upsilon=\eiota$. For any such
  $\upsilon$, the pair $(\omega^\upsilon,\beta^\upsilon)$ is an ordinary
  SBT pair, with $\omega^\upsilon=\eiota(\beta^\upsilon)^{-1}$.
  For every $3$-cycle $\tau$, the cycle $\tau$ is applicable to $\beta$ if
  and only if $\tau^\upsilon$ is applicable to $\beta^\upsilon$, and in
  that case
  \[
    \bigl((\omega\tau^{-1})^\upsilon,(\tau\beta)^\upsilon\bigr)
    =\bigl(\omega^\upsilon(\tau^\upsilon)^{-1},
      \tau^\upsilon\beta^\upsilon\bigr).
  \]
\end{proposition}

\begin{proof}
  All $N$-cycles are conjugate, so such a $\upsilon$ exists. Conjugating
  $\rho=\omega\beta$ gives
  $\eiota=\omega^\upsilon\beta^\upsilon$. Applying the same renaming of
  symbols preserves cyclic order and therefore applicability. Conjugating a
  product conjugates each factor, which gives the displayed identity.
\end{proof}

We extend the definitions of cycle orientation, annotated cycle type,
$\mu$-move, and qualifying $2$-move from Subsection~\ref{interaction} by
replacing $(\spi,\epi)$ with $(\omega,\beta)$. An applicable $3$-cycle
$\tau$ sends $(\omega,\beta)$ to $(\omega\tau^{-1},\tau\beta)$ and leaves
the target unchanged, since
$(\omega\tau^{-1})(\tau\beta)=\rho$. A sorting sequence takes $\beta$
to $\rho$ by such moves.

Let $z$ be a fixed point of $\omega$, and set $p=\beta^{-1}(z)$. Since
$\rho=\omega\beta$, both $\beta$ and $\rho$ contain the edge $p\to z$.
We call it the \emph{fixed-point edge}. Its \emph{contraction} deletes $z$
from both cycles to form a shorter cyclic-target pair. Renaming the target to
$\eiota$ turns this edge into an adjacency and commutes with deletion. Hence
Proposition~\ref{prop:contraction-preserves-distance} shows that contraction
preserves transposition distance.

To extend extended-toric equivalence to a cyclic-target pair
$(\omega,\beta)$ with target $\rho$, choose $\upsilon$ such that
$\rho^\upsilon=\eiota$ and associate the class
${(\beta^\upsilon)^\circ}_{\mathrm{ext}}$ with the pair. If $\eta$ is
another such renaming, then
$\eta\upsilon^{-1}\in C_{S_N}(\eiota)=\langle\eiota\rangle$, and hence
$\beta^\eta=(\beta^\upsilon)^{\eta\upsilon^{-1}}$ lies in the same class.
The associated class is therefore well-defined. Two cyclic-target pairs are
\emph{extended-torically equivalent} when their associated classes agree.

\section{Fixed-content words and twisted bracelets}
\label{sec:twisted-bracelets}

Fix the current cycle $\beta=(0\ 1\ \cdots\ N-1)$. An encoding word records
the supports and cyclic orders of an algebraic permutation $\omega$ relative
to $\beta$. The condition that $\omega\beta$ is an $N$-cycle identifies the
words that represent cyclic-target pairs. Rotation and twisted reflection
together with the auxiliary cycle symmetries then identify precisely the
extended-toric classes. We develop these steps in order, prove the orbit
correspondence, and then derive direct generation and a Burnside identity.

\subsection{Encoding permutations}

Let $\mathcal P$ be a fixed-point-free annotated cycle type. Expand its
multiplicities and write
$\mathcal P=\bigl((\lambda_1)_{\varepsilon_1},\ldots,
(\lambda_\ell)_{\varepsilon_\ell}\bigr)$, and set
$N:=\sum_{i=1}^{\ell}\lambda_i$, where $\lambda_i\geq2$,
$\varepsilon_i\in\{\mathrm{u},\mathrm{o}\}$, and
$\varepsilon_i=\mathrm{o}$ requires $\lambda_i\geq3$, as all $2$-cycles are
unoriented. The indices distinguish repeated annotated parts only
temporarily.

Associate a color $c_i$ with part $i$. If
$\varepsilon_i=\mathrm{u}$, cycle $i$ contributes $\lambda_i$ copies of
the unranked symbol $c_i$. If $\varepsilon_i=\mathrm{o}$, it contributes
one copy of $(c_i,h)$ for each $h\in\{0,1,\ldots,\lambda_i-1\}$.
The annotated type therefore prescribes the content of the word.

An \emph{encoding word of type $\mathcal P$} is a word with this fixed
content in which no oriented color $c_i$ has occurrence ranks
\begin{equation}
  h,h-1,\ldots,h-\lambda_i+1\pmod{\lambda_i}
  \label{eq:forbidden-unoriented-ranks}
\end{equation}
for any $h\in\mathbb Z_{\lambda_i}$ when its positions are read from left to
right. These cyclically descending sequences encode the unoriented cycle
induced by $\beta^{-1}$ on that support, contradicting
$\varepsilon_i=\mathrm{o}$. Let
$\mathcal W(\mathcal P)$ be the set of encoding words of type $\mathcal P$.
We also call them \emph{raw encoding words} because their auxiliary color
names and rank origins have not been quotiented out.

For each annotated part $m_{\varepsilon}$, a permutation of the indices
$i$ satisfying $(\lambda_i)_{\varepsilon_i}=m_{\varepsilon}$ acts by permuting
their colors. Each oriented color has the additional cyclic action
\((c_i,k)\mapsto(c_i,k+h_i)\) for
$h_i\in\mathbb Z_{\lambda_i}$. These actions together form the
\emph{auxiliary cycle group} $\mathcal A(\mathcal P)$. Thus every
$a\in\mathcal A(\mathcal P)$ has color permutations $\varphi$ within equal
annotated parts and rank shifts $h_i$ on oriented colors, and it acts by
\begin{equation*}
  a\cdot c_i=c_{\varphi(i)},
  \qquad
  a\cdot(c_i,k)=(c_{\varphi(i)},k+h_i\bmod\lambda_i).
\end{equation*}
The first formula applies to unranked symbols and the second to ranked
symbols. This action preserves $\mathcal W(\mathcal P)$.

\begin{example}[Auxiliary cycle symmetries]
  \label{ex:auxiliary-cycle-symmetries}
  Let $\mathcal P=(3_{\mathrm{o}}^2)$ and initially name its
  colors $A$ and $B$, writing ranks as subscripts. The word
  $B_2A_1B_0A_2B_1A_0$ belongs to the same
  $\mathcal A(\mathcal P)$-orbit as
  $C_0D_0C_1D_1C_2D_2$.
  The auxiliary transformation sends $B$ to $C$ with rank shift $-2$ and
  sends $A$ to $D$ with rank shift $-1$.
\end{example}

We decode $\mathbf w\in\mathcal W(\mathcal P)$ as follows. For each color $c_i$, the positions
in $\mathbf w$ occupied by symbols of color $c_i$ form the support of the cycle that
$c_i$ encodes. If $c_i$ is unoriented, let
$p_{i,0}<p_{i,1}<\cdots<p_{i,\lambda_i-1}$ be its positions, and define
$\Cycle_i(\mathbf w):=(p_{i,0}\ p_{i,\lambda_i-1}\ p_{i,\lambda_i-2}\ \cdots\ p_{i,1})$.
This is the cycle on that support in the cyclic order induced by
$\beta^{-1}$.
If $c_i$ is oriented, let $p_{i,k}$ be the unique position in $\mathbf w$ carrying
$(c_i,k)$, and define $\Cycle_i(\mathbf w):=(p_{i,0}\ p_{i,1}\ \cdots\ p_{i,\lambda_i-1})$.
Finally, set $\omega_{\mathbf w}:=\prod_{i=1}^{\ell}\Cycle_i(\mathbf w)$.
The cycles have disjoint supports, so the product is independent of how its
factors are arranged. Both words
in Example~\ref{ex:auxiliary-cycle-symmetries} therefore decode to
$\omega_{\mathbf w}=(0\ 2\ 4)(1\ 3\ 5)$ when
$\beta=(0\ 1\ \cdots\ 5)$.

\begin{proposition}[Fixed-content encoding modulo auxiliaries]
  \label{prop:fixed-content-encoding}
  The decoder is constant on the $\mathcal A(\mathcal P)$-orbits and induces
  a bijection from $\mathcal W(\mathcal P)/\mathcal A(\mathcal P)$ to the
  fixed-point-free permutations of annotated cycle type $\mathcal P$
  relative to $\beta$.
\end{proposition}

\begin{proof}
  Every unranked color decodes to the unique unoriented cyclic order on its
  support.  For every ranked color, the successive ranks determine its cycle;
  excluding
  Equation~\eqref{eq:forbidden-unoriented-ranks} makes that cycle oriented.
  Thus $\omega_{\mathbf w}$ has the prescribed annotated type.

  Conversely, let $\omega$ have annotated type $\mathcal P$. Its unoriented
  cycles are determined by their supports. Assign the available equal-part
  colors to its cycles and assign successive ranks around each oriented cycle.
  This produces an encoding word $\mathbf w$ with
  $\omega_{\mathbf w}=\omega$. Two choices differ exactly by permutations of
  equal-part colors and cyclic shifts of oriented rank origins. Consequently,
  the set $\{\mathbf w\in\mathcal W(\mathcal P):
  \omega_{\mathbf w}=\omega\}$ is a single $\mathcal A(\mathcal P)$-orbit.
\end{proof}

Set $\rho_{\mathbf w}:=\omega_{\mathbf w}\beta$. An encoding word is
\emph{realizable} when $\rho_{\mathbf w}$ is an $N$-cycle, and define
$\mathcal V(\mathcal P):=\{\mathbf w\in\mathcal W(\mathcal P):
\rho_{\mathbf w}\text{ is an }N\text{-cycle}\}$.

These are exactly the words that decode to cyclic-target pairs of annotated
type $\mathcal P$. Proposition~\ref{prop:ordinary-target-renaming} renames
each such pair to an ordinary SBT pair. In particular,
$\mathcal V(\mathcal P)$ is empty when $\mathcal P$ does not occur in $A_N$.

\begin{example}[A realizable and a nonrealizable word]
  \label{ex:word-realizability}
  Let
  $\mathcal P=(2_{\mathrm{u}}^2)$ and
  $\beta=(0\ 1\ 2\ 3)$.  For $\mathbf w=AABB$,
  $\omega_{\mathbf w}=(0\ 1)(2\ 3)$ and
  $\rho_{\mathbf w}=(0)(1\ 3)(2)$, so $\mathbf w$ is not realizable.  By
  contrast, $\mathbf v=ABAB$ gives
  $\omega_{\mathbf v}=(0\ 2)(1\ 3)$ and
  $\rho_{\mathbf v}=(0\ 3\ 2\ 1)$, and hence
  $\mathbf v\in\mathcal V(\mathcal P)$.
\end{example}

\subsection{The auxiliary--dihedral action}

Define an involution $\theta$ on symbols by fixing every unranked symbol
and setting $\theta(c_i,k):=(c_i,-k\bmod\lambda_i)$ for every oriented
color. For $\nu\in\mathbb Z_N$ and $\mathbf w\in\mathcal W(\mathcal P)$,
define
\begin{equation}
  \label{eq:twisted-word-action}
  (r_\nu\cdot\mathbf w)_j:=w_{j+\nu},
  \qquad
  (s\cdot\mathbf w)_j:=\theta(w_{-j})
  \qquad(j\in\mathbb Z_N).
\end{equation}
Thus \(s\) reflects word positions by \(j\mapsto-j\) and additionally
inverts the ranks attached to symbols of oriented colors;
unranked symbols are otherwise unchanged.

\begin{proposition}[Twisted dihedral action]
  \label{prop:twisted-dihedral-action}
  The operations in Equation~\eqref{eq:twisted-word-action} preserve
  $\mathcal W(\mathcal P)$ and satisfy
  $r_\nu r_\xi=r_{\nu+\xi}$, $r_0=\operatorname{id}$,
  $s^2=\operatorname{id}$, and $s r_\nu s=r_{-\nu}$ for
  $\nu,\xi\in\mathbb Z_N$.
  They therefore define an action of $D_N$ on $\mathcal W(\mathcal P)$.
\end{proposition}

\begin{proof}
  Rotation changes only the starting position from which occurrences are read
  around the word, while twisted reflection reverses both the position and
  rank orders. Hence both
  preserve the prescribed content and avoidance of the forbidden order in
  Equation~\eqref{eq:forbidden-unoriented-ranks}. Direct substitution gives
  $r_\nu\cdot(r_\xi\cdot\mathbf w)
  =(w_{j+\nu+\xi})_j=r_{\nu+\xi}\cdot\mathbf w$,
  $r_0\cdot\mathbf w=(w_j)_j=\mathbf w$, and
  $s\cdot(s\cdot\mathbf w)=(\theta^2(w_j))_j=\mathbf w$.
  It also gives $s\cdot\bigl(r_\nu\cdot(s\cdot\mathbf w)\bigr)
  =(w_{j-\nu})_j=r_{-\nu}\cdot\mathbf w$.
  The first two identities give $r_1^N=r_0=\operatorname{id}$, and the
  remaining identities give the other defining relations of $D_N$.
\end{proof}

\begin{proposition}[Auxiliary--dihedral action]
  \label{prop:auxiliary-dihedral-action}
  The action of $D_N$ on $\mathcal W(\mathcal P)$ normalizes
  $\mathcal A(\mathcal P)$. The corresponding semidirect product has
  underlying set $\mathcal A(\mathcal P)\times D_N$ and multiplication
  $(a,g)(a',g'):=\bigl(a(a')^g,gg'\bigr)$. Define
  $\mathcal G(\mathcal P):=\mathcal A(\mathcal P)\rtimes D_N$ with this
  multiplication. It acts on $\mathcal W(\mathcal P)$ by
  $(a,g)\cdot\mathbf w:=a\cdot(g\cdot\mathbf w)$. Equivalently, $D_N$ acts on
  $\mathcal W(\mathcal P)/\mathcal A(\mathcal P)$.
\end{proposition}

\begin{proof}
  Let $a\in\mathcal A(\mathcal P)$ have color permutation $\varphi$ and
  oriented rank shifts $h_i$. Rotations commute with $a$ because $a$ changes
  symbols without changing their positions, so $r_\nu a r_{-\nu}=a$.
  Because $s^{-1}=s$, conjugation by $s$ acts on an oriented symbol as
  $(c_i,k)\mapsto(c_i,-k)\mapsto
  (c_{\varphi(i)},-k+h_i)\mapsto(c_{\varphi(i)},k-h_i)$.
  On an unranked symbol, it only applies $\varphi$. Thus $sas^{-1}$ has color
  permutation $\varphi$ and rank shifts $-h_i$. Since the rotations and $s$
  generate $D_N$, we have $a^g=gag^{-1}\in\mathcal A(\mathcal P)$ for every
  $g\in D_N$, so conjugation defines an action of $D_N$ on
  $\mathcal A(\mathcal P)$. The identity $ga'=(a')^g g$ now gives
  $aga'g'=a(a')^g gg'$, which is the stated multiplication rule.
  It also gives $g\cdot(a'\cdot\mathbf w)
  =(a')^g\cdot(g\cdot\mathbf w)$. Hence $D_N$ acts on the auxiliary
  orbits, and
  $(a,g)\cdot\bigl((a',g')\cdot\mathbf w\bigr)
  =\bigl((a,g)(a',g')\bigr)\cdot\mathbf w$, as required.
\end{proof}

\begin{proposition}[Decoder equivariance]
  \label{prop:decoder-equivariance}
  The assignments $r_\nu\cdot\omega:=\omega^{\beta^{-\nu}}$ and
  $s\cdot\omega:=(\omega^{-1})^{\Reflect}$ for $\nu\in\mathbb Z_N$ define
  an action of $D_N$ on the fixed-point-free permutations of annotated
  cycle type $\mathcal P$ relative to $\beta$. Under this action, the
  bijection of Proposition~\ref{prop:fixed-content-encoding} from the
  $\mathcal A(\mathcal P)$-orbits of encoding words is $D_N$-equivariant.
  Equivalently,
  for every $\mathbf w\in\mathcal W(\mathcal P)$ and
  $\nu\in\mathbb Z_N$, we have
  $\omega_{r_\nu\cdot\mathbf w}=\omega_{\mathbf w}^{\beta^{-\nu}}$ and
  $\omega_{s\cdot\mathbf w}=(\omega_{\mathbf w}^{-1})^{\Reflect}$.
\end{proposition}

\begin{proof}
  Let $\Cycle_i(\mathbf w)$ denote the cycle decoded from color $c_i$.
  Position rotation replaces both the support and the defining cyclic writing
  of this cycle by their images under $\beta^{-\nu}$. Hence
  $\Cycle_i(r_\nu\cdot\mathbf w)=
  \Cycle_i(\mathbf w)^{\beta^{-\nu}}$ for every $i$.
  If $c_i$ is oriented and
  $\Cycle_i(\mathbf w)=(p_0\ p_1\ \dots\ p_{m-1})$, then the rank inversion
  in the twisted reflection gives
  $\Cycle_i(s\cdot\mathbf w)=(-p_0\ -p_{m-1}\ \dots\ -p_1)
  =(\Cycle_i(\mathbf w)^{-1})^{\Reflect}$.
  If $c_i$ is unoriented, $\Cycle_i(\mathbf w)$ is the restriction of the
  cyclic order $\beta^{-1}$ to its support. Its inverse has the restricted
  order $\beta$, and conjugation by $\Reflect$ changes this to the restricted
  order $\beta^{-1}$ on the reflected support. Thus the same equality holds
  for unoriented colors. Taking products over the disjoint supports yields
  the asserted identities. It also shows that both operations preserve
  the annotated cycle type.

  Since $\Reflect^2=\iota$ and
  $\Reflect\beta^\nu=\beta^{-\nu}\Reflect$, these operations satisfy
  $r_\nu r_\xi=r_{\nu+\xi}$, $s^2=\operatorname{id}$, and
  $sr_\nu s=r_{-\nu}$. They therefore define the asserted action. Finally,
  Proposition~\ref{prop:fixed-content-encoding} makes the decoder constant
  on $\mathcal A(\mathcal P)$-orbits, and the two identities show that the
  induced bijection intertwines the generators $r_1$ and $s$.
\end{proof}

\begin{corollary}[Realizability invariance]
  \label{cor:realizability-invariance}
  For every $a\in\mathcal A(\mathcal P)$,
  $\mathbf w\in\mathcal W(\mathcal P)$, and $\nu\in\mathbb Z_N$, we have
  $\rho_{a\cdot\mathbf w}=\rho_{\mathbf w}$,
  $\rho_{r_\nu\cdot\mathbf w}=\rho_{\mathbf w}^{\beta^{-\nu}}$, and
  $(\rho_{s\cdot\mathbf w})^\beta
  =(\rho_{\mathbf w}^{-1})^{\Reflect}$.
  Consequently, $\mathcal V(\mathcal P)$ is invariant under
  $\mathcal G(\mathcal P)$, and
  $\mathcal V(\mathcal P)/\mathcal A(\mathcal P)$ is invariant under $D_N$.
\end{corollary}

\begin{proof}
  The auxiliary equality follows from
  $\omega_{a\cdot\mathbf w}=\omega_{\mathbf w}$. Since $\beta^{-\nu}$
  centralizes $\beta$, rotation conjugates
  $\rho_{\mathbf w}=\omega_{\mathbf w}\beta$.  Also
  \[
    (\rho_{\mathbf w}^{-1})^{\Reflect}
    =(\beta^{-1}\omega_{\mathbf w}^{-1})^{\Reflect}
    =\beta(\omega_{\mathbf w}^{-1})^{\Reflect}
    =\beta\rho_{s\cdot\mathbf w}\beta^{-1}.
  \]
  Conjugation and inversion preserve cycle type, so the target is an
  $N$-cycle before either transformation if and only if it is an $N$-cycle
  afterward.
\end{proof}

A \emph{twisted bracelet of type $\mathcal P$} is an orbit of
$\mathcal G(\mathcal P)$ on $\mathcal W(\mathcal P)$. Equivalently, it is a
$D_N$-orbit on
$\mathcal W(\mathcal P)/\mathcal A(\mathcal P)$. It is \emph{realizable}
when its words belong to $\mathcal V(\mathcal P)$, which is well-defined by
Corollary~\ref{cor:realizability-invariance}.
A \emph{candidate bracelet} is a twisted bracelet considered before the
realizability condition is imposed.

\begin{remark}[Algebraic permutations and cycle graphs]
  Silva et al.~\cite[Appendix~2]{silva2022new} showed that the cycles of the
  algebraic permutation $\spi$ are in a length-preserving bijection with the
  alternating cycles of the cycle graph $G(\pi)$.  For a cyclic-target pair,
  Proposition~\ref{prop:ordinary-target-renaming} first renames the target
  to $\eiota$; deleting the anchor from the resulting current cycle then
  gives an ordinary permutation $\pi$ whose cycle graph is $G(\pi)$.
\end{remark}

\begin{theorem}[Bracelet--extended-toric-class correspondence]
  \label{thm:bracelet-class-correspondence}
  The map $\mathbf w\mapsto(\omega_{\mathbf w},\beta)$ induces a bijection
  from realizable
  twisted bracelets of type $\mathcal P$ to the extended-toric
  equivalence classes of cyclic-target pairs of annotated type
  $\mathcal P$.
\end{theorem}

\begin{proof}
  Proposition~\ref{prop:fixed-content-encoding} shows that auxiliary
  transformations do not change the reconstructed cyclic-target pair.
  Position rotation conjugates that pair by a
  power of $\beta$ and therefore realizes toric rotation. For a reflected
  rotation with index $\nu\in\mathbb Z_N$, the renaming is
  $\beta^\nu\Reflect$, which sends $x$ to
  $\nu-x\pmod N$.
  Proposition~\ref{prop:decoder-equivariance} and
  Corollary~\ref{cor:realizability-invariance} give
  $\omega'=(\omega_{\mathbf w}^{-1})^{\beta^\nu\Reflect}$, and its target
  $\rho'$ satisfies
  $(\rho')^\beta=(\rho_{\mathbf w}^{-1})^{\beta^\nu\Reflect}$.
  If $\rho_{\mathbf w}^\upsilon=\eiota$ and
  $v=\beta^\nu\Reflect\upsilon^{-1}\Reflect$, then direct substitution gives
  $F(\beta^\upsilon)^v=\beta$ and
  $\eiota^v=(\rho_{\mathbf w}^{-1})^{\beta^\nu\Reflect}$.
  Thus the transformed word realizes extended-toric reflection after
  the target is renamed to $\eiota$.  Any two choices of such a renaming
  differ only by a toric rotation.

  Conversely, apply a common renaming to a cyclic-target pair so that its
  current cycle is $\beta$, and use
  Proposition~\ref{prop:fixed-content-encoding} to obtain its unique
  auxiliary orbit of encoding words. Two such renamings differ by
  $C_{S_N}(\beta)=\langle\beta\rangle$, giving a position rotation.
  The auxiliary action gives exactly the remaining choices of equal-part
  colors and oriented rank origins. The reflected transformations established
  above give exactly the remaining identifications.
\end{proof}

\begin{example}[Rotation]
  \label{ex:rotation}
  Let $\mathcal P=(3_{\mathrm{u}}^3)$ and
  $\beta=(0\ 1\ \cdots\ 8)$. Write $A$, $B$, and $C$ for the three
  unoriented colors, and consider the encoding word
  $\mathbf w=ABACBCACB$. One position rotation gives
  $r_1\cdot\mathbf w=BACBCACBA$.
  Direct decoding gives
  $\omega_{\mathbf w}=(0\ 6\ 2)(1\ 8\ 4)(3\ 7\ 5)$ and
  $\omega_{r_1\cdot\mathbf w}=(0\ 7\ 3)(1\ 8\ 5)(2\ 6\ 4)
  =\omega_{\mathbf w}^{\beta^{-1}}$.
  The targets $\rho_{\mathbf w}=(0\ 8\ 6\ 5\ 2\ 7\ 4\ 3\ 1)$ and
  $\rho_{r_1\cdot\mathbf w}=(0\ 8\ 7\ 5\ 4\ 1\ 6\ 3\ 2)$ are both
  $9$-cycles. Thus both words are realizable and, since $r_1\in D_N$ acts
  through $\mathcal G(\mathcal P)$, they lie in the same twisted bracelet.
\end{example}

Figure~\ref{fig:rotation-bracelet} depicts the rotation directly on circular
representatives of the two words.

\begin{figure}[H]
  \centering
  \begin{tikzpicture}[
      braceletbead/.style={circle,minimum size=8mm,inner sep=0pt,
        text=white,font=\small,line width=0.5pt,scale=0.80}
    ]
    \begin{scope}[shift={(-3.2,0)}]
      \draw[gray!70,dashed,line width=0.5pt] (0,0) -- (0,1.75);
      \draw[gray!55,line width=0.5pt] (0,0) circle (1.4cm);
      \foreach \angle/\beadtext/\beadcolor in
        {90/{A}/cyclefive,50/{B}/cyclethree,10/{A}/cyclefive,
         -30/{C}/cyclethird,-70/{B}/cyclethree,-110/{C}/cyclethird,
         -150/{A}/cyclefive,-190/{C}/cyclethird,-230/{B}/cyclethree} {
        \node[braceletbead,fill=\beadcolor,draw=\beadcolor!65!black]
          at (\angle:1.4cm) {$\beadtext$};
      }
      \node[font=\small] at (0,-1.95) {$\mathbf w$};
    \end{scope}

    \begin{scope}[shift={(3.2,0)}]
      \draw[gray!70,dashed,line width=0.5pt] (0,0) -- (0,1.75);
      \draw[gray!55,line width=0.5pt] (0,0) circle (1.4cm);
      \foreach \angle/\beadtext/\beadcolor in
        {90/{B}/cyclethree,50/{A}/cyclefive,10/{C}/cyclethird,
         -30/{B}/cyclethree,-70/{C}/cyclethird,-110/{A}/cyclefive,
         -150/{C}/cyclethird,-190/{B}/cyclethree,-230/{A}/cyclefive} {
        \node[braceletbead,fill=\beadcolor,draw=\beadcolor!65!black]
          at (\angle:1.4cm) {$\beadtext$};
      }
      \node[font=\small] at (0,-1.95) {$r_1\cdot\mathbf w$};
    \end{scope}

    \draw[gray!70,->,line width=0.6pt] (-1.45,0.55) --
      node[above,font=\small] {$r_1$} (1.45,0.55);
    \node[fill=white,inner sep=2pt,font=\small] at (0,-0.3)
      {$j\mapsto j-1\pmod 9$};
  \end{tikzpicture}
  \caption{Two representatives of the same twisted bracelet. Blue,
    vermillion, and bluish green encode $A$, $B$, and $C$, respectively.
    Rotation moves every bead one position counterclockwise. The dashed
    radius marks position $0$.}
  \label{fig:rotation-bracelet}
\end{figure}
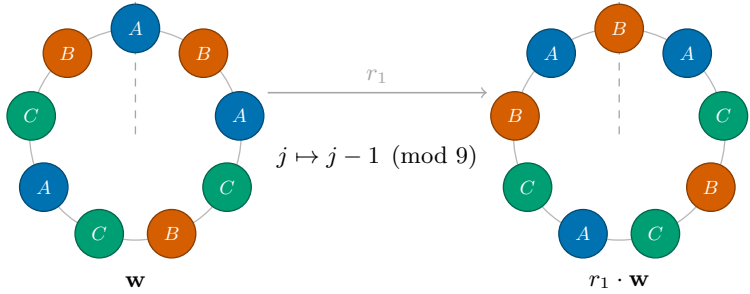

Rename the symbols in $(\omega_{\mathbf w},\beta)$ and
$(\omega_{r_1\cdot\mathbf w},\beta)$ so that both targets become $\eiota$.
The respective renamings may be chosen as
$\upsilon=(0\ 8)(1\ 7\ 4\ 5\ 2\ 3\ 6)$ and
$\upsilon'=(0\ 8)(1\ 4\ 3\ 6\ 5\ 2\ 7)$. Indeed,
$\rho_{\mathbf w}^{\upsilon}
=\rho_{r_1\cdot\mathbf w}^{\upsilon'}=\eiota$, and the resulting current
cycles are $\epi=(0\ 8\ 7\ 3\ 6\ 5\ 2\ 1\ 4)$ and
$\epi'=(0\ 8\ 4\ 7\ 6\ 3\ 2\ 5\ 1)$. Moreover,
$\epi'=R_1(\epi)=\epi^{\eiota}$, so the two anchored cycles are torically
equivalent through one rotation. Deleting the anchor gives the ordinary
permutations $\pi=[8\ 7\ 3\ 6\ 5\ 2\ 1\ 4]$ and
$\pi'=[8\ 4\ 7\ 6\ 3\ 2\ 5\ 1]$.

\begin{figure}[H]
  \centering
  \begin{tikzpicture}[x=0.46cm,y=0.46cm,transform shape]
    \foreach \i/\lab in
      {0/{+0},1/{-8},2/{+8},3/{-7},4/{+7},5/{-3},6/{+3},7/{-6},
       8/{+6},9/{-5},10/{+5},11/{-2},12/{+2},13/{-1},14/{+1},
       15/{-4},16/{+4},17/{-9}} {
      \coordinate (v\i) at (\i,0);
      \node[below=2pt,font=\scriptsize] at (v\i) {$\lab$};
    }
    \foreach \i/\j in {0/1,4/5,12/13} {
      \draw[fivecycle,line width=2pt] (v\i) -- (v\j);
    }
    \foreach \i/\j in {2/3,8/9,16/17} {
      \draw[threecycle,line width=2pt] (v\i) -- (v\j);
    }
    \foreach \i/\j in {6/7,10/11,14/15} {
      \draw[thirdcycle,line width=2pt] (v\i) -- (v\j);
    }
    \draw[fivecycle,line width=0.8pt] (v0) .. controls +(0,3.5) and +(0,3.5) .. (v13);
    \draw[fivecycle,line width=0.8pt] (v12) .. controls +(0,2.0) and +(0,2.0) .. (v5);
    \draw[fivecycle,line width=0.8pt] (v4) .. controls +(0,1.2) and +(0,1.2) .. (v1);
    \draw[threecycle,line width=0.8pt] (v2) .. controls +(0,4.0) and +(0,4.0) .. (v17);
    \draw[threecycle,line width=0.8pt] (v16) .. controls +(0,2.0) and +(0,2.0) .. (v9);
    \draw[threecycle,line width=0.8pt] (v8) .. controls +(0,1.6) and +(0,1.6) .. (v3);
    \draw[thirdcycle,line width=0.8pt] (v6) .. controls +(0,2.6) and +(0,2.6) .. (v15);
    \draw[thirdcycle,line width=0.8pt] (v14) .. controls +(0,1.2) and +(0,1.2) .. (v11);
    \draw[thirdcycle,line width=0.8pt] (v10) .. controls +(0,1.2) and +(0,1.2) .. (v7);
    \node[anchor=west,font=\small] at (0,4.5)
      {$\text{(a)}\quad G([8\ 7\ 3\ 6\ 5\ 2\ 1\ 4])$};
  \end{tikzpicture}

  \smallskip

  \begin{tikzpicture}[x=0.46cm,y=0.46cm,transform shape]
    \foreach \i/\lab in
      {0/{+0},1/{-8},2/{+8},3/{-4},4/{+4},5/{-7},6/{+7},7/{-6},
       8/{+6},9/{-3},10/{+3},11/{-2},12/{+2},13/{-5},14/{+5},
       15/{-1},16/{+1},17/{-9}} {
      \coordinate (v\i) at (\i,0);
      \node[below=2pt,font=\scriptsize] at (v\i) {$\lab$};
    }
    \foreach \i/\j in {2/3,10/11,16/17} {
      \draw[fivecycle,line width=2pt] (v\i) -- (v\j);
    }
    \foreach \i/\j in {0/1,6/7,14/15} {
      \draw[threecycle,line width=2pt] (v\i) -- (v\j);
    }
    \foreach \i/\j in {4/5,8/9,12/13} {
      \draw[thirdcycle,line width=2pt] (v\i) -- (v\j);
    }
    \draw[fivecycle,line width=0.8pt] (v2) .. controls +(0,4.0) and +(0,4.0) .. (v17);
    \draw[fivecycle,line width=0.8pt] (v16) .. controls +(0,1.6) and +(0,1.6) .. (v11);
    \draw[fivecycle,line width=0.8pt] (v10) .. controls +(0,2.0) and +(0,2.0) .. (v3);
    \draw[threecycle,line width=0.8pt] (v0) .. controls +(0,3.5) and +(0,3.5) .. (v15);
    \draw[threecycle,line width=0.8pt] (v14) .. controls +(0,2.1) and +(0,2.1) .. (v7);
    \draw[threecycle,line width=0.8pt] (v6) .. controls +(0,1.6) and +(0,1.6) .. (v1);
    \draw[thirdcycle,line width=0.8pt] (v4) .. controls +(0,2.6) and +(0,2.6) .. (v13);
    \draw[thirdcycle,line width=0.8pt] (v12) .. controls +(0,1.2) and +(0,1.2) .. (v9);
    \draw[thirdcycle,line width=0.8pt] (v8) .. controls +(0,1.2) and +(0,1.2) .. (v5);
    \node[anchor=west,font=\small] at (0,4.5)
      {$\text{(b)}\quad G([8\ 4\ 7\ 6\ 3\ 2\ 5\ 1])$};
  \end{tikzpicture}
  \caption{Cycle graphs of the ordinary permutations obtained after
    renaming the targets to $\eiota$. The thick baseline and thin arched
    edges are the black and gray edges, respectively. The colors are expository. The toric
    rotation from panel~(a) to panel~(b) adds $1$ modulo $9$ to the
    underlying vertex labels.}
  \label{fig:rotation-cycle-graphs}
\end{figure}
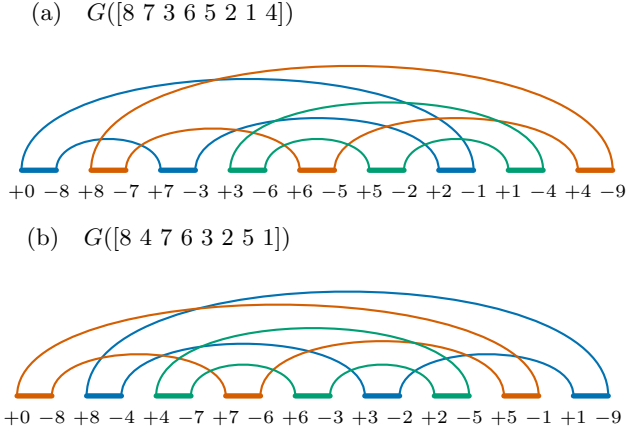

\begin{example}[Twisted reflection]
  \label{ex:twisted-reflection}
  Let $\mathcal P=(5_{\mathrm{o}},3_{\mathrm{u}})$ and
  $\beta=(0\ 1\ \cdots\ 7)$.  Write $A$ for the oriented color and $B$
  for the unoriented color, and consider the encoding word
  $\mathbf w=B A_0 B A_2 A_4 B A_1 A_3$.
  Direct decoding gives
  $\omega_{\mathbf w}=(0\ 5\ 2)(1\ 6\ 3\ 7\ 4)$.
  Twisted reflection gives
  $s\cdot\mathbf w=B A_2 A_4 B A_1 A_3 B A_0$.
  Direct decoding therefore gives
  $\omega_{s\cdot\mathbf w}=(0\ 6\ 3)(1\ 5\ 2\ 7\ 4)
  =(\omega_{\mathbf w}^{-1})^{\Reflect}$,
  where $\Reflect(x)=-x\pmod 8$. Thus rank inversion makes decoding commute with
  reflection. The targets $\rho_{\mathbf w}=(0\ 6\ 4\ 2\ 7\ 5\ 3\ 1)$ and
  $\rho_{s\cdot\mathbf w}=(0\ 5\ 3\ 1\ 7\ 6\ 4\ 2)$ are both $8$-cycles, so
  both words are realizable. Since $s\in D_N$ acts through
  $\mathcal G(\mathcal P)$, they lie in the same twisted bracelet.
\end{example}

Figure~\ref{fig:twisted-reflection-bracelet} depicts the action directly on
circular representatives of the two words.

\begin{figure}[H]
  \centering
  \begin{tikzpicture}[
      braceletbead/.style={circle,minimum size=8mm,inner sep=0pt,
        text=white,font=\small,line width=0.5pt,scale=0.80}
    ]
    \begin{scope}[shift={(-3.2,0)}]
      \draw[gray!70,dashed,line width=0.5pt] (0,-1.7) -- (0,1.7);
      \draw[gray!55,line width=0.5pt] (0,0) circle (1.4cm);
      \foreach \angle/\beadtext/\beadcolor in
        {90/{B}/cyclethree,45/{A_0}/cyclefive,
         0/{B}/cyclethree,-45/{A_2}/cyclefive,
         -90/{A_4}/cyclefive,-135/{B}/cyclethree,
         -180/{A_1}/cyclefive,-225/{A_3}/cyclefive} {
        \node[braceletbead,fill=\beadcolor,draw=\beadcolor!65!black]
          at (\angle:1.4cm) {$\beadtext$};
      }
      \node[font=\small] at (0,-1.95) {$\mathbf w$};
    \end{scope}

    \begin{scope}[shift={(3.2,0)}]
      \draw[gray!70,dashed,line width=0.5pt] (0,-1.7) -- (0,1.7);
      \draw[gray!55,line width=0.5pt] (0,0) circle (1.4cm);
      \foreach \angle/\beadtext/\beadcolor in
        {90/{B}/cyclethree,45/{A_2}/cyclefive,
         0/{A_4}/cyclefive,-45/{B}/cyclethree,
         -90/{A_1}/cyclefive,-135/{A_3}/cyclefive,
         -180/{B}/cyclethree,-225/{A_0}/cyclefive} {
        \node[braceletbead,fill=\beadcolor,draw=\beadcolor!65!black]
          at (\angle:1.4cm) {$\beadtext$};
      }
      \node[font=\small] at (0,-1.95) {$s\cdot\mathbf w$};
    \end{scope}

    \draw[gray!70,->,line width=0.6pt] (-1.45,0.55) --
      node[above,font=\small] {$s$} (1.45,0.55);
    \node[fill=white,inner sep=2pt,font=\small,align=center] at (0,-0.3)
      {$j\mapsto-j\pmod 8$\\[-1pt]$k\mapsto-k\pmod 5$};
  \end{tikzpicture}
  \caption{Two representatives of the same twisted bracelet. Blue and
    vermillion encode $A$ and $B$; reflection reverses positions and inverts
    the $A$-ranks. The dashed diameter in each panel is the reflection axis through position $0$.}
  \label{fig:twisted-reflection-bracelet}
\end{figure}
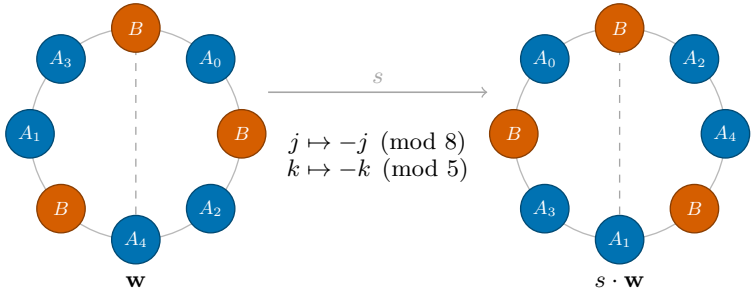

Rename the symbols in $(\omega_{\mathbf w},\beta)$ and
$(\omega_{s\cdot\mathbf w},\beta)$ so that both targets become $\eiota$.
The respective renamings may be chosen as
$\upsilon=\cycle{0,4,6,5,1,3,2,7}$ and
$\upsilon'=\cycle{1,3,2,7,4,6,5}$, respectively.  Indeed,
$\rho_{\mathbf w}^{\upsilon}
=\rho_{s\cdot\mathbf w}^{\upsilon'}=\eiota$, and the resulting current
cycles are $\epi=(0\ 4\ 3\ 7\ 2\ 6\ 1\ 5)$ and
$\epi'=(0\ 3\ 7\ 2\ 6\ 1\ 5\ 4)$.
Moreover, $\epi'=F(\epi)$, so the two anchored cycles are
extended-torically equivalent through reflection.  Deleting the anchor gives
the ordinary permutations $\pi=[4\ 3\ 7\ 2\ 6\ 1\ 5]$ and
$\pi'=[3\ 7\ 2\ 6\ 1\ 5\ 4]$.

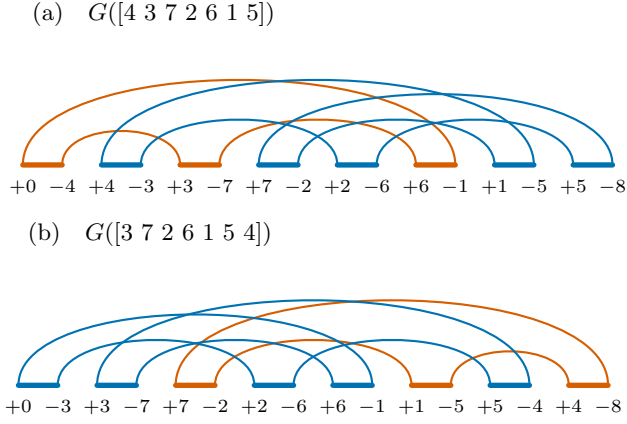
\begin{figure}[H]
  \centering
  \begin{tikzpicture}[x=0.52cm,y=0.50cm,transform shape]
    \foreach \i/\lab in
      {0/{+0},1/{-4},2/{+4},3/{-3},4/{+3},5/{-7},6/{+7},7/{-2},
       8/{+2},9/{-6},10/{+6},11/{-1},12/{+1},13/{-5},14/{+5},15/{-8}} {
      \coordinate (v\i) at (\i,0);
      \node[below=2pt,font=\scriptsize] at (v\i) {$\lab$};
    }
    \foreach \i/\j in {0/1,4/5,10/11} {
      \draw[threecycle,line width=2pt] (v\i) -- (v\j);
    }
    \foreach \i/\j in {2/3,6/7,8/9,12/13,14/15} {
      \draw[fivecycle,line width=2pt] (v\i) -- (v\j);
    }
    \draw[threecycle,line width=0.8pt] (v0) .. controls +(0,3.0) and +(0,3.0) .. (v11);
    \draw[threecycle,line width=0.8pt] (v4) .. controls +(0,1.2) and +(0,1.2) .. (v1);
    \draw[threecycle,line width=0.8pt] (v10) .. controls +(0,1.6) and +(0,1.6) .. (v5);
    \draw[fivecycle,line width=0.8pt] (v12) .. controls +(0,1.6) and +(0,1.6) .. (v7);
    \draw[fivecycle,line width=0.8pt] (v8) .. controls +(0,1.6) and +(0,1.6) .. (v3);
    \draw[fivecycle,line width=0.8pt] (v2) .. controls +(0,3.0) and +(0,3.0) .. (v13);
    \draw[fivecycle,line width=0.8pt] (v14) .. controls +(0,1.6) and +(0,1.6) .. (v9);
    \draw[fivecycle,line width=0.8pt] (v6) .. controls +(0,2.5) and +(0,2.5) .. (v15);
    \node[anchor=west,font=\small] at (0,4.0)
      {$\text{(a)}\quad G([4\ 3\ 7\ 2\ 6\ 1\ 5])$};
  \end{tikzpicture}

  \smallskip

  \begin{tikzpicture}[x=0.52cm,y=0.50cm,transform shape]
    \foreach \i/\lab in
      {0/{+0},1/{-3},2/{+3},3/{-7},4/{+7},5/{-2},6/{+2},7/{-6},
       8/{+6},9/{-1},10/{+1},11/{-5},12/{+5},13/{-4},14/{+4},15/{-8}} {
      \coordinate (v\i) at (\i,0);
      \node[below=2pt,font=\scriptsize] at (v\i) {$\lab$};
    }
    \foreach \i/\j in {4/5,10/11,14/15} {
      \draw[threecycle,line width=2pt] (v\i) -- (v\j);
    }
    \foreach \i/\j in {0/1,2/3,6/7,8/9,12/13} {
      \draw[fivecycle,line width=2pt] (v\i) -- (v\j);
    }
    \draw[threecycle,line width=0.8pt] (v10) .. controls +(0,1.6) and +(0,1.6) .. (v5);
    \draw[threecycle,line width=0.8pt] (v14) .. controls +(0,1.2) and +(0,1.2) .. (v11);
    \draw[threecycle,line width=0.8pt] (v4) .. controls +(0,3.0) and +(0,3.0) .. (v15);
    \draw[fivecycle,line width=0.8pt] (v0) .. controls +(0,2.5) and +(0,2.5) .. (v9);
    \draw[fivecycle,line width=0.8pt] (v6) .. controls +(0,1.6) and +(0,1.6) .. (v1);
    \draw[fivecycle,line width=0.8pt] (v2) .. controls +(0,3.0) and +(0,3.0) .. (v13);
    \draw[fivecycle,line width=0.8pt] (v12) .. controls +(0,1.6) and +(0,1.6) .. (v7);
    \draw[fivecycle,line width=0.8pt] (v8) .. controls +(0,1.6) and +(0,1.6) .. (v3);
    \node[anchor=west,font=\small] at (0,4.0)
      {$\text{(b)}\quad G([3\ 7\ 2\ 6\ 1\ 5\ 4])$};
  \end{tikzpicture}
  \caption{Cycle graphs of the ordinary permutations obtained after
    renaming the targets to $\eiota$.  The thick baseline and thin arched
    edges are the black and gray edges, respectively.  The colors are expository.  If the
    vertex labels are ignored, panel~(b) is the horizontal reflection of
    panel~(a).}
  \label{fig:twisted-reflection-cycle-graphs}
\end{figure}

\begin{corollary}[Distance invariance]
  \label{cor:bracelet-distance-invariance}
  Transposition distance is constant on every realizable twisted
  bracelet.
\end{corollary}

\begin{proof}
  This follows from Theorem~\ref{thm:bracelet-class-correspondence} and
  Propositions~\ref{prop:rotation-invariance},
  \ref{prop:reflection-preserves-sorting}, and
  \ref{prop:ordinary-target-renaming}.
\end{proof}

\subsection{Direct generation of realizable representatives}

We adapt the fixed-content necklace and bracelet recursions of
Sawada~\cite{Sawada2003} and Karim et al.~\cite{Karim2013}, reviewed in
Appendix~\ref{app:fixed-content-generation}, to generate realizable twisted
bracelet representatives.

Fix a total order on the colors. For an oriented color, order its ranks by
$0<1<\cdots<\lambda_i-1$, and extend the resulting alphabet order
lexicographically to words.
For an encoding word $\mathbf v$, define its \emph{auxiliary
standardization} $\operatorname{std}(\mathbf v)$ by reading from left to
right. At the first occurrence of a color, assign it the least unused color
with the same annotated part. For an oriented color, also choose the unique
rank shift that sends the first encountered rank to $0$. Use the same color
assignment and rank shift at every later occurrence.
The resulting word $\operatorname{std}(\mathbf v)$ is the lexicographically
least word in $\mathcal A(\mathcal P)\cdot\mathbf v$.

If $\mathcal A(\mathcal P)$ is trivial, then twisted reflection is ordinary
reversal, so we use the bracelet recursion of Karim et
al.~\cite{Karim2013}. Otherwise, we use Sawada's necklace
recursion~\cite{Sawada2003} and enforce
$\mathbf w\leq\operatorname{std}(g\cdot\mathbf w)$ for every $g\in D_N$
during prefix generation. Compare the two words from left to right as their
symbols become determined. Reject the branch when
$\operatorname{std}(g\cdot\mathbf w)$ first becomes smaller than $\mathbf w$,
and stop that comparison when it first becomes larger. This avoids enumerating
the auxiliary group or storing previously encountered orbits.

During either recursion, reject a branch when an oriented color completes a
rank sequence of the form in
Equation~\eqref{eq:forbidden-unoriented-ranks}.
Represent a permutation $\tau$ by the directed edges
$z\mapsto\tau(z)$. If every completion $\mathbf w$ of a prefix satisfies
$\omega_{\mathbf w}(x)=y$, then
$\rho_{\mathbf w}(\beta^{-1}(x))=y$. Call the directed edge
$\beta^{-1}(x)\mapsto y$ a \emph{forced target arc}. Reject a prefix as soon
as its forced target arcs contain a directed cycle on fewer than $N$ vertices,
because no completion can then make $\rho_{\mathbf w}$ an $N$-cycle. Every
complete word surviving these tests is emitted immediately.

\begin{theorem}[Correctness of direct generation]
  \label{thm:twisted-generation-correctness}
  For every fixed-point-free annotated cycle type $\mathcal P$, direct
  generation emits exactly one word from each realizable twisted bracelet of
  type $\mathcal P$, namely its least word.
\end{theorem}

\begin{proof}
  Let $\mathbf w$ be the least word in a realizable twisted bracelet. Since
  $D_N\leq\mathcal G(\mathcal P)$, it is least among its rotations and
  rank-inverting reflections with fixed color names and rank origins. The
  appropriate fixed-content recursion therefore reaches $\mathbf w$.
  Membership in $\mathcal W(\mathcal P)$ ensures that it passes the forbidden
  rank-sequence test. By construction, every
  $\operatorname{std}(g\cdot\mathbf w)$ lies in the twisted bracelet
  containing $\mathbf w$ and is no smaller than $\mathbf w$. Hence no orbit
  comparison rejects its prefixes. Its forced target arcs lie in the single cycle
  $\rho_{\mathbf w}$ and cannot contain a directed cycle on fewer than $N$
  vertices. Thus $\mathbf w$ is emitted.

  Conversely, let $\mathbf v$ be emitted. The fixed-content and forbidden
  rank-sequence conditions give $\mathbf v\in\mathcal W(\mathcal P)$. Its
  completed comparisons, or the bracelet recursion when
  $\mathcal A(\mathcal P)$ is trivial, give
  $\mathbf v\leq\operatorname{std}(g\cdot\mathbf v)$ for every $g\in D_N$.
  Since standardization selects the least auxiliary image, $\mathbf v$ is
  least in its $\mathcal G(\mathcal P)$-orbit. At completion, the forced
  target arcs are exactly the arcs of $\rho_{\mathbf v}$. If
  $\rho_{\mathbf v}$ had more than one cycle, each would use fewer than $N$
  vertices, contrary to the pruning rule. Hence $\rho_{\mathbf v}$ is an
  $N$-cycle and $\mathbf v$ is realizable. A finite orbit has a unique least
  word, which proves the result.
\end{proof}

\subsubsection{Generation cost}

When a symbol of maximum prescribed multiplicity is the greatest alphabet
symbol, Sawada and Karim et al.~\cite{Sawada2003,Karim2013} prove
constant-amortized-time bounds, excluding output writing, for their classical
fixed-content necklace and bracelet generators. Because our generator also
performs realizability and auxiliary tests, these bounds do not carry over.

We now record how the auxiliary action affects the generation cost. Call an
index $i$ \emph{affected} if $\varepsilon_i=\mathrm{o}$ or if there is an
index $j\ne i$ with
$(\lambda_j)_{\varepsilon_j}=(\lambda_i)_{\varepsilon_i}$. These are exactly
the indices for which some element of $\mathcal A(\mathcal P)$ changes a
symbol of color $c_i$. Define
$S:=\sum_{\{i:\,i\text{ is affected}\}}\lambda_i$.
Every encoding word has exactly $S$ positions occupied by affected colors.
Consequently, $0\leq S\leq N$, the auxiliary action is trivial exactly when
$S=0$, and $S=N$ can occur.

When $S>0$, choose an affected index $i_\star$ and make $c_{i_\star}$ the
least color. The prenecklace recursion fixes this color, and rank $0$ when it
is oriented, as the first symbol. Standardization preserves annotated parts,
so standardizing a rotation or twisted reflection that begins with another
annotated part yields a word whose first symbol is larger than that of the
original word. For each position whose color has annotated part
$(\lambda_{i_\star})_{\varepsilon_{i_\star}}$, compare the word with its
standardizations obtained by reading from that position forward and backward
around the cycle, with rank inversion in the backward reading. All these
positions are among the $S$ affected positions, so there are at most $2S$
comparisons.

Call a prefix \emph{stopped} when the recursion either emits it as a
completed word or pruning rejects it. A root-to-stopped-prefix path has at
most $N$ extensions. If each comparison retains the length already matched,
it examines at most $N$ pairs consisting of a symbol of the word and the
symbol in the same position of its standardized forward or backward reading.
Maintaining their states during extension and backtracking requires $O(S)$
time per extension.
Thus the auxiliary comparisons require $O(NS)$ time per path and no time
when $S=0$. When $S=N$, this gives a worst-case bound of $O(N^2)$ per path.

For a conservative total bound, let
$U(\mathcal P):=N!/\prod_{\{i:\varepsilon_i=\mathrm{u}\}}\lambda_i!$
be the number of fixed-content words before the orientation restrictions are
imposed. Stopped prefixes form an antichain.
Appending unused symbols in a fixed order injects them into the fixed-content
words, so there are at most $U(\mathcal P)$ of them.
Summing the costs of all root-to-stopped-prefix paths gives an upper bound,
since work at a prefix is counted once for each path containing it.
Hence auxiliary comparisons take $O(NS\ U(\mathcal P))$ time. The incremental
realizability test takes $O(N\log N)$ time per path.
The remaining work, including prefix-reflection tests, admits a bound of $O(N)$
time per extension and hence $O(N^2)$ time per path. Since $S\leq N$,
total sequential time is $O(N^2U(\mathcal P))$, including output writing but
excluding any subsequent use of the emitted words.

\subsection{Counting identities for encoding words and twisted bracelets}

For $m\geq2$ and $\varepsilon\in\{\mathrm{u},\mathrm{o}\}$, let
$q_{m,\varepsilon}$ be the multiplicity of $m_{\varepsilon}$ in
$\mathcal P$, with $q_{2,\mathrm{o}}=0$.

\begin{proposition}[Raw words and auxiliary classes]
  \label{prop:encoding-word-count}
  The numbers of encoding words and auxiliary orbits are
  \begin{align}
    |\mathcal W(\mathcal P)|
    =
    \frac{N!}{
    \displaystyle\prod_{m\geq2}
    (m!)^{q_{m,\mathrm{u}}+q_{m,\mathrm{o}}}
    }
    \displaystyle\prod_{m\geq3}
    (m!-m)^{q_{m,\mathrm{o}}},
    \label{eq:raw-encoding-word-count}\\
    |\mathcal W(\mathcal P)/\mathcal A(\mathcal P)|
    =
    \frac{N!}{
    \displaystyle\prod_{m\geq2}
    (m!)^{q_{m,\mathrm{u}}+q_{m,\mathrm{o}}}
    q_{m,\mathrm{u}}!\,q_{m,\mathrm{o}}!
    }
    \displaystyle\prod_{m\geq3}
    \bigl((m-1)!-1\bigr)^{q_{m,\mathrm{o}}}.
    \label{eq:auxiliary-encoding-class-count}
  \end{align}
\end{proposition}

\begin{proof}
  First treat all colors as distinguished. Partition the $N$ positions into
  their supports, giving the factorial quotient in
  Equation~\eqref{eq:raw-encoding-word-count}. An unoriented support has one
  admissible filling. The $m$ ranked symbols of an oriented color have $m!$
  orders, of which the $m$ cyclically descending orders in
  Equation~\eqref{eq:forbidden-unoriented-ranks} are excluded. This proves the
  raw-word count.

  The auxiliary action is free because every color occurs and every rank of
  an oriented color occurs. Its group has order
  $|\mathcal A(\mathcal P)|=\prod_{m\geq2}q_{m,\mathrm{u}}!
  \prod_{m\geq3}m^{q_{m,\mathrm{o}}}q_{m,\mathrm{o}}!$.
  Dividing the raw-word count by this order and using
  $(m!-m)/m=(m-1)!-1$ proves
  Equation~\eqref{eq:auxiliary-encoding-class-count}.
\end{proof}

\begin{proposition}[Orientation refinement of the Doignon--Labarre~\cite{doignon2007hultman} count]
  \label{prop:recover-doignon-labarre}
  Let $\lambda$ be a fixed-point-free cycle type on $N$ symbols, and let
  $H_N(\lambda)$ be the number of permutations in $S_{N-1}$ whose cycle
  graph has alternating-cycle lengths $\lambda$. Then
  \begin{equation}
    H_N(\lambda)
    =
    \sum_{\substack{\mathcal P\text{ has underlying}\\
                    \text{cycle type }\lambda}}
    \left|\mathcal V(\mathcal P)/\mathcal A(\mathcal P)\right|.
    \label{eq:recover-doignon-labarre}
  \end{equation}
\end{proposition}

\begin{proof}
  Fix $\beta=(0\ 1\ \cdots\ N-1)$. The bijection of Doignon and
  Labarre~\cite[Theorem~10]{doignon2007hultman} identifies the permutations
  counted by $H_N(\lambda)$ with the permutations $\sigma\in S_N$ of cycle
  type $\lambda$ for which $\beta\sigma^{-1}$ is an $N$-cycle. Apply the
  cycle-type-preserving bijection $\sigma\mapsto\omega:=\sigma^{-1}$. The
  target conditions are equivalent because
  \[
    \sigma^{-1}(\beta\sigma^{-1})\sigma
    =\sigma^{-1}\beta
    =\omega\beta.
  \]
  Hence this is equivalently the set of permutations $\omega$ of cycle type
  $\lambda$ for which $\omega\beta$ is an $N$-cycle.

  The annotated cycle types $\mathcal P$ with underlying cycle type
  $\lambda$ partition this set. Proposition~\ref{prop:fixed-content-encoding}
  identifies the part indexed by $\mathcal P$ with
  $\mathcal V(\mathcal P)/\mathcal A(\mathcal P)$. Summing over the parts
  proves Equation~\eqref{eq:recover-doignon-labarre}.
\end{proof}

Equation~\eqref{eq:recover-doignon-labarre} separates the Doignon--Labarre
count by cycle orientations. The following result then quotients each
summand by extended-toric equivalence.

\begin{corollary}[Burnside identity for twisted bracelets]
  \label{cor:twisted-bracelet-burnside}
  The number $B_{\mathrm{real}}(\mathcal P)$ of realizable twisted
  bracelets of type $\mathcal P$ is
  \[
    B_{\mathrm{real}}(\mathcal P)
    =
    \frac{1}{2N}\sum_{\nu\in\mathbb Z_N}
    \left(
    |\Fix_{\mathcal V(\mathcal P)/\mathcal A(\mathcal P)}(r_\nu)|
    +|\Fix_{\mathcal V(\mathcal P)/\mathcal A(\mathcal P)}(r_\nu s)|
    \right).
  \]
\end{corollary}

\begin{proof}
  Apply Lemma~\ref{lem:burnside} to the action of $D_N$ on
  $\mathcal V(\mathcal P)/\mathcal A(\mathcal P)$ from
  Proposition~\ref{prop:auxiliary-dihedral-action} and
  Corollary~\ref{cor:realizability-invariance}.
\end{proof}

\begin{example}[One oriented $5$-cycle and two unoriented $3$-cycles]
  \label{ex:five-three-three-burnside}
  Let $\mathcal P=(5_{\mathrm{o}},3_{\mathrm{u}}^2)$.
  Here $N=11$. Choose the support of the $5$-cycle, one of its
  $(5-1)!-1=23$ oriented cyclic orders, and an unordered partition of the
  remaining six symbols into two $3$-element supports. Each such support
  has a unique unoriented cyclic order relative to $\beta^{-1}$. Thus
  Equation~\eqref{eq:auxiliary-encoding-class-count} gives
  $\binom{11}{5}\,23\,\frac{1}{2}\binom{6}{3}=106260$
  auxiliary classes of encoding words. Independent enumeration of these
  classes and the target-cycle test gives the fixed-set counts in
  Table~\ref{tab:five-three-three-fixed-sets}.
  \begin{table}[ht]
    \centering
    \caption{Fixed-set counts for the action of $D_{11}$ on the auxiliary
    classes of type $\mathcal P=(5_{\mathrm{o}},3_{\mathrm{u}}^2)$.}
    \label{tab:five-three-three-fixed-sets}
    \begin{tabular}{@{}lrr@{}}
      \toprule
      Elements of $D_{11}$ & Number of group elements
                             & $\lvert\Fix_{\mathcal V(\mathcal P)/\mathcal A(\mathcal P)}(g)\rvert$ \\
      \midrule
      Identity               & $1$                                            & $9702$ \\
      Nonidentity rotations  & $10$                                           & $0$    \\
      Reflected rotations    & $11$                                           & $76$   \\
      \bottomrule
    \end{tabular}
  \end{table}
  To justify the zero entry, if an auxiliary class is fixed by
  $r_\nu$, its decoded permutation commutes with $\beta^\nu$. Since
  $11$ is prime and $\nu\ne0$,
  $C_{S_{11}}(\beta^\nu)=\langle\beta\rangle$, and no element of this cyclic
  group has cycle type $(5,3^2)$. All reflections in $D_{11}$ are
  conjugate, so their fixed sets have the same size. Therefore
  Corollary~\ref{cor:twisted-bracelet-burnside} gives
  $B_{\mathrm{real}}(\mathcal P)=\frac{9702+11\cdot76}{22}
  =479$.
  Thus there are exactly $479$ extended-toric equivalence classes of
  cyclic-target pairs whose algebraic permutation has one oriented $5$-cycle
  and two unoriented $3$-cycles.
\end{example}

\section{The transposition diameter of
  \texorpdfstring{$S_{16}$}{S16}}
\label{sec:td16}

We now apply the twisted-bracelet framework to the last undetermined diameter
through $S_{17}$.  Throughout this section, $N=17$ and an instance in
$S_{16}$ is represented by the ordinary SBT pair $(\spi,\epi)$. The ambient
space has $16!$ instances, so we combine fixed-point contraction with reductions of the ambient space based on cycle structure and direct
generation at the quotient level. The computer-assisted part uses two
independent computations. One emits sorting witnesses for the required
twisted bracelets, and a separate checker reconstructs the routing and
enumeration and replays every witness. Their certificate format is the only
interface. The source and instructions permit both certificate regeneration
and cheaper independent checking~\cite{sourcecode}.

We route every instance through disjoint cases. If $\spi$ has a fixed point,
Propositions~\ref{prop:adjacency-fixed-point}
and~\ref{prop:contraction-preserves-distance} reduce the problem to an
instance in $S_{15}$, so $TD(15)=9$ gives
$d_t(\epi)\leq9$~\cite{eriksson2001sorting}. We may therefore assume that $\spi$ is fixed-point-free. Let
$\lambda(\spi)$ denote its cycle-length partition of $17$. Since $\spi\in
A_{17}$, $\lambda(\spi)$ has an even number of even parts. There are exactly
$33$ such partitions. After attaching orientation, with equal parts
indistinguishable and $2$-cycles necessarily unoriented, they give $194$
annotated cycle types.

Among the oriented odd-length cycles of length at least $5$, select one of
minimum length. An \emph{even-pair reduction} replaces two even-length cycles
of lengths $x$ and $y$ in one move by a fixed point and a cycle of length
$x+y-1$. It preserves the target because the move changes $(\omega,\beta)$ to
$(\omega\tau^{-1},\tau\beta)$. Table~\ref{tab:td16-root-routing} is a priority
routing, so a type enters the first applicable row.

\begin{table}[ht]
  \centering
  \caption{Priority routing of the fixed-point-free annotated cycle types.}
  \label{tab:td16-root-routing}
  \small
  \begin{tabularx}{\textwidth}{@{}cXrX@{}}
    \toprule
    Route & First applicable condition & Types & Treatment \\
    \midrule
    $\mathbf U$   & every cycle is unoriented
      & $33$ & direct generation and sorting search \\
    $\mathbf T$   & an oriented $3$-cycle is present
      & $74$ & Corollary~\ref{cor:oriented-three-cycles} \\
    $\mathbf E_9$ & the selected length is $9$, and an even-length cycle is
      oriented
      & $3$ & even-pair reduction \\
    $\mathbf O$   & a selected odd-length cycle remains
      & $47$ & odd-length-cycle reduction \\
    $\mathbf E$   & none of the preceding conditions holds
      & $37$ & even-pair reduction \\
    \midrule
    \multicolumn{2}{@{}l}{Total} & $194$ & \\
    \bottomrule
  \end{tabularx}
\end{table}

\begin{proposition}[Annotated-cycle-type routing]
  \label{prop:td16-root-routing}
  The five routes in Table~\ref{tab:td16-root-routing} are disjoint and
  exhaust the $194$ fixed-point-free annotated cycle types.
\end{proposition}

\begin{proof}
  The priority rule gives disjointness. If $m_k$ is the multiplicity of part
  $k$, then a partition has $\prod_{k\geq3}(m_k+1)$ annotations: for each
  $k\geq3$, any number of the $m_k$ indistinguishable cycles may be oriented,
  whereas $2$-cycles are unoriented. Summing over the $33$ partitions gives
  $194$, and applying the conditions in order gives
  $33,74,3,47,37$ types, whose sum is $194$.

  The three types in route $\mathbf E_9$ are
  $(9_{\mathrm{o}},6_{\mathrm{o}},2_{\mathrm{u}})$,
  $(9_{\mathrm{o}},4_{\mathrm{o}},4_{\mathrm{u}})$, and
  $(9_{\mathrm{o}},4_{\mathrm{o}}^2)$. A type reaching route
  $\mathbf E$ is not all-unoriented but has no oriented odd-length cycle.
  It has an oriented even-length cycle, and the parity of $\spi$ guarantees at
  least two even parts, so an even-pair reduction is available.
\end{proof}

For route $\mathbf U$, set
$p=\norm{\spi}$, equivalently
$p=\sum_{k\in\lambda(\spi)}\lfloor k/2\rfloor$. In
Table~\ref{tab:td16-unoriented}, $2^a$ means that the part $2$ occurs $a$
times. The last column counts extended-toric classes of cyclic-target pairs.

\begin{table}[ht]
  \centering
  \caption{Route $\mathbf U$: all-unoriented fixed-point-free cycle types in
    $A_{17}$.}
  \label{tab:td16-unoriented}
  \footnotesize
  \begin{tabularx}{\textwidth}{@{}c>{\raggedright\arraybackslash}Xr@{}}
    \toprule
    $p$ & Partitions $\lambda$ & Extended-toric classes \\
    \midrule
    --- & $(17)$; $(13,2^2)$; $(12,3,2)$; $(11,4,2)$; $(11,3^2)$;
          $(10,5,2)$; $(10,4,3)$; $(9,6,2)$; $(9,5,3)$; $(9,4^2)$;
          $(9,2^4)$; $(8,3,2^3)$ & $0$ \\
    $8$ & $(8,7,2)$; $(8,6,3)$; $(8,5,4)$; $(7,6,4)$;
          $(7,4,2^3)$; $(6^2,5)$;
          $(6,5,2^3)$; $(6,4,3,2^2)$; $(5,4^2,2^2)$; $(5,2^6)$;
          $(4^3,3,2)$; $(4,3,2^5)$ & $1\,612\,777$ \\
    $7$ & $(7^2,3)$; $(7,5^2)$; $(7,3^2,2^2)$;
          $(6,3^3,2)$; $(5^2,3,2^2)$; $(5,4,3^2,2)$;
          $(4^2,3^3)$; $(3^3,2^4)$ & $2\,182\,743$ \\
    $6$ & $(5,3^4)$ & $63\,364$ \\
    \bottomrule
  \end{tabularx}
\end{table}

Direct generation found no realizable word in the first row. For each of the
$3\,858\,884$ extended-toric classes in the remaining rows, the certificate
contains a sorting sequence of length at most $9$. The checker independently
regenerates its representative and replays the sequence.

In route $\mathbf T$, Corollary~\ref{cor:oriented-three-cycles} clears an
oriented $3$-cycle in one move. Contraction of the three resulting fixed
points and $TD(13)=8$ give $d_t(\epi)\leq1+8=9$ for all $74$ types in this
route~\cite{eriksson2001sorting}.

For a cyclic-target pair $(\omega,\beta)$ and a cycle $\Cycle$ of its
algebraic permutation $\omega$, let $\beta_{\Cycle}$ be the cycle obtained
from $\beta$ by deleting the symbols outside $\Supp(\Cycle)$. We call
$(\Cycle,\beta_{\Cycle})$ the \emph{local configuration induced by
$(\omega,\beta)$ on $\Supp(\Cycle)$}. A $3$-cycle whose support is contained
in $\Supp(\Cycle)$ is applicable to $\beta$ if and only if it is applicable
to $\beta_{\Cycle}$. Moreover, its move type and whether it creates a fixed
point are determined by replacing $\Cycle$ with $\Cycle\tau^{-1}$, since all
other cycles of $\omega$ are unchanged.

Call a sequence of moves $\tau_1,\dots,\tau_t$ \emph{support-clearing} for
$\Cycle$ if $\Supp(\tau_i)\subseteq\Supp(\Cycle)$ and $\tau_i$ is applicable
at its step for every $i$, and if the resulting algebraic permutation fixes
every symbol of $\Supp(\Cycle)$. Deleting the symbols outside
$\Supp(\Cycle)$ at each step shows that such a sequence is determined by the
induced local configuration. Length $15$ does not occur
because the only fixed-point-free partition of $17$ containing that part is
$(15,2)$, whose permutations are odd.

For two oriented cycles $\Cycle$ and $\Cycle'$ of the same length $k$, choose
bijections from their supports to $\mathbb Z_k$ that send
$\beta_{\Cycle}$ and $\beta_{\Cycle'}$, respectively, to
$\beta_k=(0\ 1\ \cdots\ k-1)$, and apply each bijection to both components of
its local configuration. The configurations are equivalent under
\emph{twisted-bracelet symmetry on their supports} when the resulting
encoding words lie in the same auxiliary--dihedral orbit from
Section~\ref{sec:twisted-bracelets}.

For a $k$-cycle $\gamma$ and $u,v\in\Supp(\gamma)$, let the \emph{directed
cyclic distance} $\operatorname{dist}_{\gamma}(u,v)$ be the unique integer in
$\{0,1,\dots,k-1\}$ such that
$\gamma^{\operatorname{dist}_{\gamma}(u,v)}(u)=v$.

\begin{lemma}[Qualifying-move criterion]
  \label{lem:qualifying-criterion}
  Let $(\omega,\beta)$ be a cyclic-target pair and $\Cycle$ an odd $k$-cycle
  of $\omega$, where $k\geq3$. Suppose that, for some
  $a\in\Supp(\Cycle)$ and $m\in\{2,4,\dots,k-1\}$, the symbols $a$,
  $\Cycle(a)$, and $\Cycle^m(a)$ occur in this cyclic order in
  $\beta_{\Cycle}$. Then $\tau=(a\ \Cycle(a)\ \Cycle^m(a))$ is a
  qualifying $2$-move with
  $\Supp(\tau)\subseteq\Supp(\Cycle)$.
\end{lemma}

\begin{proof}
  Let $b=\Cycle(a)$ and $c=\Cycle^m(a)$. Write
  $\Cycle=(a\ b\ Y\ c\ Z)$, where $\lvert Y\rvert=m-2$ and
  $\lvert Z\rvert=k-m-1$. The cyclic-order condition makes $\tau$
  applicable, and direct multiplication gives
  $\Cycle\tau^{-1}=(a\ Z)(b)(c\ Y)$.
  Both $\lvert Y\rvert$ and $\lvert Z\rvert$ are even because $m$ is even
  and $k$ is odd. Thus the odd cycle $\Cycle$ is replaced by three odd
  cycles, including the fixed point $(b)$, so $\tau$ is a qualifying
  $2$-move.
\end{proof}

\begin{corollary}
  \label{cor:half-turn}
  Let $(\omega,\beta)$ be a cyclic-target pair and $\Cycle$ an odd $k$-cycle
  of $\omega$, where $k\geq3$. If no qualifying $2$-move $\tau$ satisfies
  $\Supp(\tau)\subseteq\Supp(\Cycle)$, then
  \[
    \operatorname{dist}_{\beta_{\Cycle}}(a,\Cycle(a))
    \geq\frac{k+1}{2}
    \qquad\text{for every }a\in\Supp(\Cycle).
  \]
\end{corollary}

\begin{proof}
  Fix $a\in\Supp(\Cycle)$ and write
  $\beta_{\Cycle}=(a\ X\ \Cycle(a)\ Y)$, where $X$ and $Y$ are possibly
  empty strings. For every even $m$ with $2\leq m\leq k-1$,
  Lemma~\ref{lem:qualifying-criterion} and the hypothesis imply that $a$,
  $\Cycle(a)$, and $\Cycle^m(a)$ do not occur in this cyclic order in
  $\beta_{\Cycle}$, so $\Cycle^m(a)$ occurs in $X$. Thus $X$ contains the
  $(k-1)/2$ symbols $\Cycle^2(a),\Cycle^4(a),\dots,\Cycle^{k-1}(a)$; hence
  $\operatorname{dist}_{\beta_{\Cycle}}(a,\Cycle(a))
  =\lvert X\rvert+1\geq(k+1)/2$. The claim follows, since $a$ was arbitrary.
\end{proof}

The exceptional twisted bracelet in the following proposition is represented
by the local configuration
\begin{equation}
  \Cycle_9=(0\ 5\ 1\ 6\ 2\ 7\ 3\ 8\ 4),
  \qquad
  \beta_9=(0\ 1\ 2\ 3\ 4\ 5\ 6\ 7\ 8).
  \label{eq:exceptional-nine-cycle}
\end{equation}

\begin{proposition}[Odd-length-cycle reduction]
  \label{prop:td16-odd-cycle-reduction}
  Let $(\spi,\epi)$ be an ordinary SBT pair on $E_{17}$, and let $\Cycle$ be
  an oriented cycle of $\spi$ whose length $k$ is one of
  $5,7,9,11,13,17$. Unless the induced local configuration
  $(\Cycle,\beta_{\Cycle})$ belongs to the twisted bracelet represented by
  $(\Cycle_9,\beta_9)$ in Equation~\eqref{eq:exceptional-nine-cycle}, either
  there is a qualifying $2$-move $\tau$ with
  $\Supp(\tau)\subseteq\Supp(\Cycle)$ or there is a support-clearing sequence
  for $\Cycle$ of length $(k+1)/2$.
\end{proposition}

\begin{proof}[Computational verification]
  Relabel the support of the induced local configuration by $\mathbb Z_k$ so
  that $\beta_{\Cycle}$ becomes $\beta_k=(0\ 1\ \cdots\ k-1)$, and write the
  relabeled algebraic cycle as
  $\Cycle=(x_0\ x_1\ \cdots\ x_{k-1})$ with $x_0=0$. Every $k$-cycle has
  exactly one such writing. A depth-first search over
  $x_1,\dots,x_{k-1}$ enumerates these writings, pruning by two conditions
  that are necessary for the absence of a qualifying $2$-move. First,
  $(x_{j+1}-x_j)\bmod k\geq(k+1)/2$ for every $j$, by
  Corollary~\ref{cor:half-turn}. Second, for every $i$ and every even $m$ with
  $2\leq m\leq k-1$, the triple $x_i$, $x_{i+1}$, $x_{i+m}$ must fail the
  cyclic order of Lemma~\ref{lem:qualifying-criterion}; indices here are taken
  modulo $k$, and each instance is tested as soon as its three entries are
  determined. Consequently, every cycle not already covered by the
  qualifying-move alternative survives the pruning; the retained set may
  a priori be larger.

  For $k=5,7,9$, the searches visit $14$, $80$, and $506$ nodes,
  respectively. For $k=11,13,17$, they visit $3\,391$, $23\,619$, and
  $1\,243\,098$ nodes, respectively, and terminate in seconds. For each such
  $k$, the cycles surviving both pruning conditions are exactly the powers
  of $\beta$ listed in
  Table~\ref{tab:local-exceptions}, together with
  $\beta^{k-1}=\beta^{-1}$, which is unoriented and therefore excluded. Each
  listed power admits a support-clearing sequence of length $(k+1)/2$,
  recorded in the certificate and replayed by the checker, except
  $\beta_9^{5}$, which is the cycle of
  Equation~\eqref{eq:exceptional-nine-cycle}; an exhaustive search over all
  sequences of $(k+1)/2$ applicable $0$- and $2$-moves confirms that it admits
  none.
\end{proof}

\begin{table}[ht]
  \centering
  \caption{Route $\mathbf O$: the oriented $k$-cycles surviving both pruning
    conditions. Every one is a power of $\beta$, and
    $\Cycle_9=\beta_9^5$.}
  \label{tab:local-exceptions}
  \small
  \begin{tabular}{@{}rlr@{}}
    \toprule
    $k$ & Surviving cycles & Number \\
    \midrule
    $5$  & $\beta^{3}$                               & $1$ \\
    $7$  & $\beta^{4}$                               & $1$ \\
    $9$  & $\beta^{5}$, $\beta^{7}$                 & $2$ \\
    $11$ & $\beta^{6}$                               & $1$ \\
    $13$ & $\beta^{7}$, $\beta^{10}$, $\beta^{11}$ & $3$ \\
    $17$ & $\beta^{9}$, $\beta^{13}$, $\beta^{15}$ & $3$ \\
    \bottomrule
  \end{tabular}
\end{table}

Route $\mathbf E_9$ precedes route $\mathbf O$ because an annotated cycle type
does not determine whether its selected $9$-cycle is exceptional. Thus every
such type with an oriented even-length cycle uses an even-pair reduction. The
even-pair reduction in Proposition~\ref{prop:even-pair-reduction} creates
exactly one fixed point and joins cycles of lengths $x$ and $y$ into one of
length $x+y-1$.

For route $\mathbf O$, a support-clearing sequence closes its branch because
$(k+1)/2+TD(16-k)=9$ for $k=5,7,9,11,13$, using the known smaller
diameters~\cite{eriksson2001sorting}; for $k=17$, the support-clearing
sequence itself has length $9$. A qualifying $2$-move splits the
selected cycle into three odd-length cycles. If it creates two fixed points,
then $1+TD(14)=9$ closes the branch. If it creates one, the other two cycles
have odd lengths at least $3$ and total length $k-1$.

For the exception in route $\mathbf O$, call the other nontrivial cycles the
\emph{companions}. They must be unoriented: otherwise an oriented $3$-cycle
routes to $\mathbf T$, an oriented $5$- or $7$-cycle is selected first, or an
oriented even-length cycle routes to $\mathbf E_9$. Their partition has size
$8$, and fixed-point-freeness and parity leave exactly $(6,2)$, $(5,3)$,
$(4^2)$, and $(2^4)$. For each row of
Table~\ref{tab:exceptional-nine}, we exhaustively interleave the companion
symbols with the nine ranked symbols of $\Cycle_9$ in every order consistent
with the prescribed content. We then identify twisted bracelets before
testing realizability.

\begin{table}[ht]
  \centering
  \caption{Route $\mathbf O$: exhaustive verification of the exceptional
    oriented $9$-cycle. The third column counts extended-toric classes of
    cyclic-target pairs.}
  \label{tab:exceptional-nine}
  \small
  \begin{tabular}{@{}lrrr@{}}
    \toprule
    Annotated cycle type & Candidate bracelets & Extended-toric classes & Sorted in $\leq9$ \\
    \midrule
    $(9_{\mathrm{o}},6_{\mathrm{u}},2_{\mathrm{u}})$
      & $20\,160$ & $527$    & $527$ \\
    $(9_{\mathrm{o}},5_{\mathrm{u}},3_{\mathrm{u}})$
      & $40\,040$ & $2\,096$ & $2\,096$ \\
    $(9_{\mathrm{o}},4_{\mathrm{u}}^2)$
      & $25\,410$ & $1\,643$ & $1\,643$ \\
    $(9_{\mathrm{o}},2_{\mathrm{u}}^4)$
      & $75\,950$ & $12\,732$& $12\,732$ \\
    \bottomrule
  \end{tabular}
\end{table}

Every type in route $\mathbf E$ also admits the reduction from
Proposition~\ref{prop:even-pair-reduction}, as shown in the proof of
Proposition~\ref{prop:td16-root-routing}. The three types in $\mathbf E_9$ and
the $37$ in $\mathbf E$ therefore contribute $40$ branches that create
exactly one fixed point.

The qualifying $2$-moves from route $\mathbf O$ that create exactly one fixed
point and the even-pair reductions from routes $\mathbf E_9$ and $\mathbf E$
lead to the same remaining case. Call the resulting instance a
\emph{one-fixed-point successor}. For each successor, we recompute every
cycle-orientation annotation because the move changes the current cyclic
order.

In a one-fixed-point successor, an oriented $3$-cycle closes the branch by
Corollary~\ref{cor:oriented-three-cycles}, since the two moves already used
together with $TD(12)=7$ give $2+7=9$~\cite{eriksson2001sorting}. For every
remaining successor, we contract its fixed-point edge, directly generate
every realizable twisted bracelet of each resulting annotated cycle type on
$16$ symbols, and search within eight moves. Because different branches can
yield the same annotated cycle type after contraction, we perform each
resulting enumeration and search only once.

Both computations performed the same $52$ successor searches. They generated
$4\,511\,223\,147$ extended-toric classes, of which exact replay certified
$4\,511\,221\,111$ within eight moves. The remaining
$2\,036$ were retained conservatively as residual obstructions. Of these,
$495$ have type $(11_{\mathrm{o}},5_{\mathrm{o}})$ and $1\,541$ have type
$(9_{\mathrm{o}},7_{\mathrm{o}})$.

To avoid relying on the selected move for a residual obstruction, restore its
contracted fixed point in each of the $16$ gaps of the current cycle; the
target $\rho=\omega\beta$ contains the same fixed-point edge. Call each result
a \emph{restored successor}. Its possible reconstructions belong to two
inverse families. The \emph{odd-split} family chooses the fixed point and one
symbol from each of the residual obstruction's two nontrivial cycles,
reversing a move that split one odd-length cycle. The \emph{even-pair} family
chooses the fixed point and two symbols from one of those cycles, reversing a
move that joined two even-length cycles. Exactly one cyclic orientation of
each chosen triple gives an applicable inverse move, which preserves the
target. A reconstruction retained in the corresponding inverse family is an
\emph{avoidance preimage}.
Table~\ref{tab:td16-avoidance} gives the counts.

\begin{table}[ht]
  \centering
  \caption{Exhaustive obstruction-avoidance verification.}
  \label{tab:td16-avoidance}
  \small
  \begin{tabular}{@{}lrr@{}}
    \toprule
    Inverse family & Inverse moves tested & Avoidance-preimage classes \\
    \midrule
    Odd-split & $1\,988\,928$ & $1\,339\,101$ \\
    Even-pair & $1\,920\,192$ & $178\,439$ \\
    \midrule
    Union (zero overlap) & & $1\,517\,540$ \\
    \bottomrule
  \end{tabular}
\end{table}

For every avoidance-preimage class in this union, exact replay certifies a
sorting sequence of length at most $9$. The reconstruction move is not part
of that sequence. Hence every such class has distance at most $9$.

The checker regenerates the required representative of every
twisted bracelet and replays its recorded sequence. It verifies that every
move is applicable and preserves the target, that each sequence ends at the
identity, and that its length satisfies the bound. It also recomputes all
routing, residual-obstruction, and avoidance-preimage counts.
Any missing or invalid
witness, count discrepancy, or uncovered case causes failure. A search that
finds no sorting sequence contributes a residual obstruction handled by the
reconstruction and replay summarized in Table~\ref{tab:td16-avoidance}; any
other computational failure invalidates the certificate rather than supporting
a distance conclusion. Complete source code, per-type totals, the certificate
specification, and reproduction instructions are available
online~\cite{sourcecode}, as is the certificate~\cite{td16certificate}.
The certificate occupies approximately $92$ GB.
On a machine with four Nvidia RTX 4080S GPUs, an AMD Epyc 7C13 CPU, and
$128$ GiB of memory, certificate emission required about 7 hours and
CPU-only checking required about 2 hours. These observed wall-clock times
are provided only for reproducibility planning.

\begin{theorem}
  \label{thm:td16}
  The transposition diameter of $S_{16}$ is $TD(16)=9$.
\end{theorem}

\begin{proof}
  Fixed-point contraction settles every instance whose algebraic permutation
  has an adjacency. For a fixed-point-free instance,
  Proposition~\ref{prop:td16-root-routing} assigns its annotated cycle type to
  exactly one route. Table~\ref{tab:td16-unoriented} closes $\mathbf U$, and
  Corollary~\ref{cor:oriented-three-cycles} closes $\mathbf T$.
  Proposition~\ref{prop:td16-odd-cycle-reduction} closes the support-clearing
  branches of $\mathbf O$, while
  Table~\ref{tab:exceptional-nine} closes its sole exceptional local twisted
  bracelet.
  The remaining qualifying moves and the even-pair reductions for
  $\mathbf E_9$ and $\mathbf E$ either close by a known smaller diameter or
  produce one-fixed-point successors. The $52$ successor searches
  close all but the recorded residual obstructions, and
  Table~\ref{tab:td16-avoidance} closes every corresponding avoidance
  preimage. Thus $d_t(\epi)\leq9$ for every $\epi$ representing a permutation
  in $S_{16}$, so $TD(16)\leq9$.
  The reverse-permutation formula of Meidanis, Walter, and
  Dias~\cite{MeidanisWalterDias1997} gives
  $TD(16)\geq\lceil17/2\rceil=9$. Therefore
  $TD(16)=9$.
\end{proof}

\begin{corollary}
  \label{cor:td19}
  The transposition diameter of $S_{19}$ is $TD(19)=11$.
\end{corollary}

\begin{proof}
  The general upper bound of Eriksson et al.~\cite{eriksson2001sorting}
  gives $TD(19)\leq\lfloor(2\cdot19-2)/3\rfloor=12$. Their recurrence and
  Theorem~\ref{thm:td16} sharpen this to $TD(19)\leq TD(16)+2=11$.
  The lower bound of Elias and
  Hartman~\cite{EliasHartman2006} gives
  $TD(19)\geq\lfloor(19+1)/2\rfloor+1=11$.  Therefore $TD(19)=11$.
\end{proof}

\begin{corollary}\label{cor:upper-bound-mod3}
  For every $n\equiv 1\pmod{3}$ with $n\geq 16$,
  \[
    TD(n) \;\leq\; \left\lfloor\frac{2n-2}{3}\right\rfloor - 1.
  \]
\end{corollary}

\begin{proof}
  Write $n=16+3k$ with $k\geq0$.  Iterating the Eriksson
  recurrence~\cite{eriksson2001sorting} gives
  $TD(n)\leq TD(16)+2k=9+2k$.  Since $n\equiv1\pmod{3}$, we have
  $\lfloor(2n-2)/3\rfloor=(2n-2)/3=10+2k$, and the claim follows.
\end{proof}

\section{Conclusion}\label{sec:conclusion}

We developed the twisted-bracelet framework for cyclic-target pairs of
prescribed fixed-point-free annotated cycle type.  Its fixed-content encoding,
auxiliary cycle action, and twisted dihedral action turn extended-toric
equivalence into an orbit problem. The orbit correspondence yields a Burnside
identity for the number of extended-toric classes of cyclic-target pairs, while
the direct-generation method produces one encoding-word representative of
each class.
Together, they replace enumeration followed by symmetry reduction with a
direct description of the symmetry classes.

For $S_{16}$, we first reduced the ambient space by combining fixed-point
contraction with reductions based on cycle structure. For the remaining cases,
we employed the twisted-bracelet framework to directly generate the extended-toric
classes for exhaustive sorting search and verification. Theorem~\ref{thm:td16}
establishes $TD(16)=9$, closing the twenty-five-year gap among the exact
diameters through $n=17$.  This result also determines the next odd case.
The Eriksson recurrence~\cite{eriksson2001sorting} lowers the upper bound for
$TD(19)$ from $12$ to $11$, which meets the Elias--Hartman lower
bound~\cite{EliasHartman2006}. Corollary~\ref{cor:td19} therefore establishes
$TD(19)=11$. More generally, Corollary~\ref{cor:upper-bound-mod3} lowers the Eriksson
upper bound by one for every $n\equiv1\pmod{3}$, $n\geq16$.

Among the diameters $TD(n)$ with $n\leq19$, only $TD(18)$ remains
undetermined, making it a natural next target for a systematic approach using
the twisted-bracelet framework, similar to the one used here to determine
$TD(16)$.
Another direction concerns the general upper bound of Eriksson et al.,
$TD(n)\leq \left\lfloor\frac{2n-2}{3}\right\rfloor$ for $n\geq9$~\cite{eriksson2001sorting}.
The framework could be used to enumerate and generate extremal extended-toric classes, whose study may inform new arguments to improve this
long-standing bound.

\section*{Acknowledgments}

The methodology, mathematical results, algorithms, and computational design
were developed by the authors. OpenAI Codex and Anthropic Claude Code were used to assist with
implementation, manuscript editing, and the preparation of preliminary proof
drafts. The authors reviewed and tested all AI-assisted code and critically
revised and finalized all proof drafts. The authors independently verified
the mathematical claims and computational results. The authors assume
responsibility for all content.

\appendix

\section{Cycle graph}\label{appendix-a}

We use the standard cycle graph of Bafna and Pevzner~\cite{BafnaPevzner1998},
as recalled by Silva et al.~\cite[Appendix~A1]{silva2023barrier}.  Let
$\pi=[\pi_1\;\pi_2\;\dots\;\pi_n]\in S_n$, and set $\pi_0=0$ and
$\pi_{n+1}=n+1$.  The \emph{cycle graph} $G(\pi)$ is the directed,
edge-colored graph with vertex set
$\{+0,\allowbreak-1,\allowbreak+1,\allowbreak-2,\allowbreak+2,
\allowbreak\dots,\allowbreak-n,\allowbreak+n,\allowbreak-(n+1)\}$.
Its black edges are
$-\pi_i\longrightarrow+\pi_{i-1}$ for $1\leq i\leq n+1$, and its gray
edges are $+i\longrightarrow-(i+1)$ for $0\leq i\leq n$.  Thus the black
edges encode $\pi$, and the gray edges encode the identity permutation.

Every positive vertex has one incoming black edge and one outgoing gray
edge, and every negative vertex has one incoming gray edge and one outgoing
black edge.  Consequently, every vertex has in-degree and out-degree one,
and $G(\pi)$ decomposes uniquely into alternating directed cycles.  A
\emph{$k$-cycle} is a cycle containing $k$ black edges.  We call
it an \emph{odd-length cycle} or an \emph{even-length cycle} according as
$k$ is odd or even.
Let $c_{\mathrm{odd}}(\pi)$ be the number of odd-length cycles in
$G(\pi)$.

For $1\leq i<j<k\leq n+1$, a transposition $\tau(i,j,k)$ exchanges the
adjacent blocks $\pi_i\dots\pi_{j-1}$ and
$\pi_j\dots\pi_{k-1}$.  Its application to $\pi$ is
\[
  \tau(i,j,k)\mathbin{\cdot}\pi
  =[\pi_1\dots\pi_{i-1}\;
    \pi_j\dots\pi_{k-1}\;
    \pi_i\dots\pi_{j-1}\;
    \pi_k\dots\pi_n].
\]

The Bafna--Pevzner lower bound is the following.

\begin{theorem}[Lower bound; Bafna and
  Pevzner~\cite{BafnaPevzner1998}]
  \label{thm:cycle-graph-lower-bound}
  For every $\pi\in S_n$,
  $d_t(\pi)\geq \frac{n+1-c_{\mathrm{odd}}(\pi)}{2}$.
\end{theorem}

\section{Fixed-content necklace and bracelet recursions}
\label{app:fixed-content-generation}

This appendix briefly reviews the classical recursions that supply the
necklace and ordinary-reflection parts of our construction. Let the alphabet
be totally ordered, and let the prescribed multiplicities sum to $N$. A
nonempty word is a \emph{Lyndon word} when it is strictly smaller than each of
its nonidentity cyclic rotations. For a nonempty word $u$, let
$\operatorname{lyn}(u)$ be the length of its longest prefix that is a Lyndon
word. A \emph{prenecklace} is a prefix of a necklace representative.

The fundamental theorem of necklaces used by Sawada~\cite{Sawada2003} and
Karim et al.~\cite{Karim2013} gives the recursion in terms of this
Lyndon-prefix length. Suppose that a prenecklace
$u=u_0u_1\cdots u_{j-1}$ has length $j$, and set
$p=\operatorname{lyn}(u)$. Appending a letter $b$ gives another prenecklace
exactly when $b\geq u_{j-p}$. Equality gives
$\operatorname{lyn}(ub)=p$, whereas a strict inequality gives
$\operatorname{lyn}(ub)=j+1$. The fixed-content version considers only
letters whose remaining multiplicity is positive. At length $N$, the
prenecklace is a necklace representative exactly when $p$ divides $N$. The
optimized recursion avoids scanning letters whose multiplicities are
exhausted. When only the largest letter has positive remaining multiplicity,
the recursion appends all of its remaining copies at once. These changes
reduce the traversal without changing the generated representatives.

The Lyndon-prefix values in Sawada's~\cite[pp.~263--264]{Sawada2001} trace of
$0010023003$ illustrate this prenecklace recursion. The successive nonempty
prefixes are
\[
  0,\ 00,\ 001,\ 0010,\ 00100,\ 001002,\ 0010023,\
  00100230,\ 001002300,\ 0010023003,
\]
and their Lyndon-prefix lengths are
$1,1,3,3,3,6,7,7,7,10$. After $00100$, for example, $p=3$, so the next
letter must be at least the letter $1$ three positions earlier. Choosing $2$
changes the Lyndon-prefix length to $6$. Before the final extension, this
length is $7$, so the lower bound is $1$; choosing $3$ changes the length to
$10$, which divides the word length.

Karim et al.~\cite{Karim2013} combine Sawada's~\cite{Sawada2001}
ordinary-reflection test with this prenecklace recursion restricted to the
prescribed content. For each prenecklace prefix
$u_0u_1\cdots u_{j-1}$, compare it with its reversal
$u_{j-1}u_{j-2}\cdots u_0$. If the prefix is lexicographically larger, no
extension can be a bracelet representative, so the branch is rejected. If the
two words are equal, set $q=j$; otherwise leave $q$ unchanged. Thus, for a
completed necklace representative $u=u_0u_1\cdots u_{N-1}$, the parameter $q$
is the length of its longest palindromic prefix. The word $u$ is a bracelet
representative exactly when no prefix comparison caused rejection and
\[
  u_qu_{q+1}\cdots u_{N-1}
  \leq
  u_{N-1}u_{N-2}\cdots u_q.
\]
Karim et al.~\cite{Karim2013} accelerate the prefix comparisons using
run-length encoding and maintain the final suffix comparison incrementally,
while the remaining multiplicities restrict the recursion to the prescribed
content.

\begingroup
\sloppy
\raggedright
\bibliographystyle{siamplain}
\bibliography{manuscript}
\endgroup
\end{document}